%% file: main.tex
\documentclass{article}
\usepackage{graphicx} 

\usepackage{floatrow}
\usepackage{lipsum}
\usepackage{amsfonts}
\usepackage{graphicx}
\usepackage{epstopdf,comment}
\usepackage{algorithm}
\usepackage{algpseudocode}
\usepackage{amsmath,amsthm,amssymb}
\usepackage{mathtools}
\usepackage{graphicx}
\usepackage{enumerate,url}
\ifpdf
  \DeclareGraphicsExtensions{.eps,.pdf,.png,.jpg}
\else
  \DeclareGraphicsExtensions{.eps}
\fi

\usepackage{amsopn}

\makeatletter
\newcommand*{\addFileDependency}[1]{
  \typeout{(#1)}
  \@addtofilelist{#1}
  \IfFileExists{#1}{}{\typeout{No file #1.}}
}
\makeatother

\newtheorem{theorem}{Theorem}[section]
\newtheorem{lemma}[theorem]{Lemma}
\newtheorem{definition}[theorem]{Definition}
\newtheorem{example}[theorem]{Example}

\newtheorem{corollary}{Corollary}

\newtheorem{proposition}{Proposition}
\newtheorem{remark}{Remark}

\newcommand{\RR}{\mathbb{R}}
\newcommand{\R}{\mathbb{R}}

\newcommand{\Tn}{\mathcal{U}_m}

\newcommand{\Trop}{\text{Trop}}

\newcommand{\trtor}{\mathbb{R}^e/\mathbb{R}{\bf 1}}
\DeclareMathOperator*{\argmin}{arg\,min}
\DeclareMathOperator*{\argmax}{arg\,max}

\usepackage{tikz,tikz-cd}
\usepackage{color} 
\definecolor{mygreen}{rgb}{0, 0.82 , 0.0}
\definecolor{myorange}{rgb}{0.85, 0.5, 0.0}
\definecolor{mymagenta}{rgb}{1,0,1}
\definecolor{mycyan}{rgb}{0.0, 1.0, 1.0}
\definecolor{mypink}{rgb}{0.97, 0.68, 0.64}
\definecolor{myblue}{rgb}{0.57, 0.61, 0.90}
\definecolor{mylightgreen}{rgb}{0.66, 1, 0.66}
\definecolor{mygrey}{rgb}{0.6, 0.6, 0.6}

\title{Tropical Logistic Regression Model on Space of Phylogenetic Trees}
\author{Georgios Aliatimis \and Ruriko Yoshida \and Burak Boyaci \and James A. Grant}
\date{}

\begin{document}

\maketitle

\abstract{
\noindent
Classification of gene trees is an important task both in the analysis of multi-locus phylogenetic data, and assessment of the convergence of Markov Chain Monte Carlo (MCMC) analyses used in Bayesian phylogenetic tree reconstruction.
The logistic regression model is one of the most popular classification models in statistical learning, thanks to its computational speed and interpretability.
However, it is not appropriate to directly apply the standard logistic regression model to a set of phylogenetic trees, as the space of phylogenetic trees is non-Euclidean and thus contradicts the standard assumptions on covariates.
\\
It is well-known in tropical geometry and phylogenetics that the space of phylogenetic trees is a tropical linear space in terms of the max-plus algebra.  
Therefore, in this paper, we propose an analogue approach of the logistic regression model in the setting of tropical geometry.

Our proposed method outperforms classical logistic regression in terms of {\color{black}Area under the ROC Curve (AUC)} in numerical examples, including with
data generated by the multi-species coalescent model.
Theoretical properties such as statistical consistency 
have been proved and generalization error rates have been 
derived. 
Finally, our classification algorithm is proposed as 
an MCMC convergence criterion for \texttt{Mr Bayes}. 
Unlike {\color{black}the} convergence metric {\color{black}used by \texttt{Mr Bayes}} which is only dependent on 
tree topologies, our method is sensitive 
to branch lengths and therefore provides a more robust 
metric for convergence. In a test case, it is illustrated 
that the tropical logistic regression can differentiate
between two independently run MCMC chains, even when the 
standard metric cannot. 
  }

\section*{Introduction} \label{sec:intro}
Phylogenomics is a new field that applies tools from phylogenetics to genome datasets.  
The multi-species coalescent model is often used to model the distribution of gene trees under a given species tree \cite{maddison2008mesquite}. 
The first step in statistical analysis of phylogenomic data 
is to 
analyze sequence alignments to determine whether their evolutionary histories are congruent with each other. In this step, evolutionary biologists aim to identify genes with unusual evolutionary events, such as duplication, horizontal gene transfer, or hybridization \cite{Ane}. To accomplish this, they compare multiple sets of {\em gene trees}, that is, phylogenetic trees reconstructed from alignments of genes, with each gene tree characterised by
the aforementioned evolutionary events.
The classification of gene trees into different categories is therefore important for analyzing multi-locus phylogenetic data.

Tree classification can also help in assessing the convergence of Markov Chain Monte Carlo (MCMC) analyses for Bayesian inference on phylogenetic tree reconstruction. Often, we apply MCMC samplers to estimate the posterior distribution of a phylogenetic tree given an observed alignment. These samplers typically run multiple independent Markov chains on the same observed dataset. The goal is to check whether these chains converge to the same distribution. This process is often done by comparing summary statistics computed from sampled trees. These statistics often only depend on the tree topologies, and so they naturally lose information about the branch lengths of the sampled trees. Alternatively, we propose the use of a classification model that classifies trees from different chains and uses statistical measures such as the {\color{black}Area under the ROC Curve (AUC)} to indicate how distinguishable the two chains are. Consequently, high values of AUCs indicate that the chains have not converged to the equilibrium distribution. Currently, there is no classification model over the space of phylogenetic trees, the set of all possible phylogenetic trees with a fixed number of leaves. In this paper, we propose a classifier that is appropriate for the tree space and is sensitive to branch lengths, unlike the summary statistics of most MCMC convergence diagnostic tools.

In Euclidean geometry, the logistic regression model is
the simplest generalized linear model for classification.
It is a supervised learning method that classifies data points by modeling the log-odds of having a response variable in a particular class as a linear combination of predictors. 
This model is very popular in statistical learning 
due to its simplicity, computational speed and interpretability.
However, directly applying such classical supervised models to a set of sampled trees 
may be misleading, since the space of phylogenetic trees 
does not conform to Euclidean geometry. 

The space of phylogenetic trees with labeled leaves $[m]$ is a union of lower dimensional polyhedral cones with dimension $m - 1$ over $\mathbb{R}^e$ where $e = \binom{m}{2}$ \cite{SS,LSTY}.
This space  
is not Euclidean and even lacks convexity \cite{LSTY}. 
In fact,  \cite{SS} showed that the space of phylogenetic trees 
is a {\em tropicalization} of linear subspaces defined by 
a system of tropical linear equations
\cite{10.1093/bioinformatics/btaa564} and is therefore a tropical linear space.

Consequently, many researchers have applied tools from tropical geometry to statistical learning methods in phylogenomics,
such as principal component analysis over the {space of phylogenetic trees} with a given set of leaves $[m]$ \cite{YZZ, 10.1093/bioinformatics/btaa564},  kernel density estimation \cite{YBMH}, MCMC sampling \cite{yoshida2022hit}, and support vector machines \cite{YTMM}. 
Recently, \cite{Akian2021} proposed a tropical linear regression over the tropical projective space as the best-fit tropical hyperplane. 
However, our logistic regression model is built from first 
principles and is not a trivial extension of the aforementioned tropical regression model.

In this paper, an analog of the logistic regression is developed over the tropical projective
space, which is the quotient space $\trtor$ where ${\bf 1}:= (1, 1, \ldots, 1)$. 
Given a sample of observations within this space, the proposed model finds the 
``best-fit'' tree representative $\omega_Y \in \trtor$
of each class $Y \in \{0,1\}$
and the ``best-fit'' deviation of the 
gene trees. This tree representative is a statistical parameter
and can be interpreted as the corresponding species tree 
of the gene trees.
The deviation parameter is defined in terms of the 
variability of branch lengths of gene trees.
It is established that the median tree, specifically the 
Fermat-Weber point, can asymptotically approximate the 
inferred tree representative of each class. The response variable $Y \in \{0,1\}$ 
has conditional distribution $Y|X \sim {\rm Bernoulli}( S(h(X)))$, 
where $h(x)$ is small when $x$ is close to $\omega_0$ and far away from $\omega_{1}$ and
vice versa. 

In Section \ref{sec:basics_trees} 
an overview of tropical geometry and
its connections to phylogenetics is presented. The one-species
and two-species
tropical logistic models are developed in Section 
\ref{sec:method}. Theoretical results, including
the optimality of the proposed method over 
tropically distributed predictor trees, 
the distance distribution of those trees 
from their representative, the 
consistency of estimators and the generalization error of each 
model are stated in Section \ref{sec:method}
and proved in Supplement \ref{sec:proofs}.
Section \ref{sec:optimization}
explains
the benefit and suitability of using
the Fermat-Weber point approximation for the inferred trees
and a sufficient optimality condition is stated.
Computational results are presented in Section 
\ref{sec:results} where a toy example is considered for 
illustration purposes. Additionally, 
a comparison study between classical, tropical and BHV logistic
regression is conducted on data generated under the coalescent 
model. In both the toy example and the coalescent gene trees
example, our model outperforms the alternative regression
models.
Finally, our model is proposed as an alternative MCMC convergence criterion in Section
\ref{sec:mb}.
The paper concludes with a discussion in Section \ref{sec:discussion}.
The code developed and implemented for the proposed
model can be found in \cite{github_aliatimis}.

The dataset can be found at DRYAD with DOI: 10.5061/dryad.tht76hf65.

\section{Tropical Geometry and Phylogenetic Trees} \label{sec:basics_trees}
\subsection{Tropical Basics} \label{sec:basics}

This section covers the basics of tropical 
geometry and provides the theoretical 
background for the model developed in later 
sections. The concept of a tropical metric will be used 
when defining a suitable distribution for the gene trees.
For more details regarding the basic concepts of 
tropical geometry
covered in this section, 
readers are recommended to consult \cite{MS}.


A key tool from tropical geometry is the {\em tropical metric} also known as the {\em tropical distance} defined as follows:
\begin{definition}[Tropical distance]
\label{tropdist}
The {\em tropical distance}, more formally known as 
the Generalized Hilbert projective metric,
between two vectors 
$v, \, w \in (\mathbb R\cup\{-\infty\})^e$ 
is defined as 
\begin{equation} \label{eq:tropmetric} 
d_{\rm tr}(v,w)  := \| v-w \|_{\rm tr} =
\max_{i} \bigl\{ v_i - w_i \bigr\} - \min_{i} \bigl\{ v_i - w_i \bigr\},
\end{equation}
where $v = (v_1, \ldots , v_e)$ and $w= (w_1, \ldots , w_e)$. 
\end{definition}

\begin{remark}
Consider two vectors $v=(c,\dots,c) = 
c {\bf 1} \in \mathbb R^e$ and 
$w={\bf 0} \in \mathbb R ^ e$. 
It is easy to verify that  
\(
d_{\rm tr}(v,w) =  0 
\) 
and as a result $d_{\rm tr}$ is not a metric in 
$\mathbb R ^ e$. 
The space in which $d_{\rm tr}$ is a metric 
treats all points in 
$ \{ c {\bf 1} : c \in \mathbb R  \} = 
\mathbb{R} {\bf 1}$ as the same point. 
The quotient space 
\( 
    ( \mathbb R \cup \{-\infty\} ) ^ e /
    \mathbb R {\rm 1}
\)
achieves just that. 
\end{remark}

\begin{proposition}
\textit{
The function $d_{\rm tr}$ is a well-defined metric on 
\(
    (\mathbb R\cup\{-\infty\})^e \!/\mathbb R {\bf 1},
\)
where ${\bf 1} \in \mathbb R ^ e$ is the vector of all-ones. 
}
\end{proposition}

\subsection{Equidistant Trees and Ultrametrics}
Phylogenetic trees depict the evolutionary relationship 
between different taxa. For example, they may summarise the 
evolutionary history of certain species. 
The leaves of the tree correspond to the species studied, 
while internal nodes represent (often hypothetical) 
common ancestors of those species and their ancestors. 
In this paper, only 
rooted phylogenetic trees are considered, 
with the common ancestor of all taxa based on the root 
of the tree.
The branch lengths of these trees are measured 
in evolutionary units, i.e. the amount of evolutionary change.
Under the molecular clock hypothesis, the rate of 
genetic change between species is constant over time, 
which implies genetic equidistance and allows us to 
treat evolutionary units as proportional to time units. 
Consequently, phylogenetic trees of extant species 
are {\em equidistant trees}. 
\begin{definition}[Equidistant tree]
Let $T$ be a rooted phylogenetic tree with leaf label set $[m]$, where $m \in \mathbb{N}$ is the number of leaves.
If the distance from all leaves $i \in [m]$ to the root is the same, 
then $T$ is an equidistant tree.
\end{definition}
It is noted that the molecular 
clock hypothesis has limitations and the rate of 
genetic change can in fact vary from one species to another. 
However, the assumption that gene trees are equidistant is not 
unusual in phylogenomics; 
the multispecies coalescent model 
makes that assumption
in order to conduct inference on the 
species tree from a sample of gene trees \cite{mesquite}.
The proposed classification method is not restricted to
equidistant trees, but all coalescent model gene trees produced in 
Section \ref{sec:coalescent_model}.
are equidistant.
\par 
To conduct any mathematical analysis, 
a vector representation of trees is needed. 
A common way is to use BHV coordinates \cite{BHV} 
but in this paper  
{\em distance matrices} are used instead, which are 
then transformed into vectors.
The main reason is simplicity and computational efficiency;
it is much easier to compute gradients in the tropical projective torus than in the BHV space. 

\begin{definition}[Distance matrix]
    Consider a phylogenetic tree $T$ with leaf label set $[m]$. 
    Its distance matrix $D \in \mathbb{R}^{m \times m}$
    has components $D_{ij}$ 
    being the pairwise distance between a leaf $i \in [m]$ 
    to a leaf $j \in [m]$. 
    It follows that the matrix is symmetric with zeros on its diagonals.
    For equidistant trees, $D_{ij}$ is equal to 
    twice the difference between the current time and the latest time that
    the common ancestor of $i$ and $j$ was alive.
\end{definition}
To form a vector, 
the distance matrix $D$ is mapped onto $\mathbb{R}^e$ by 
vectorizing the strictly upper triangular part of $D$,
i.e. 
$$D \mapsto (D_{12}, \dots, D_{1m}, D_{23}, \dots, 
D_{2m}, \dots, D_{(m-1)m}) \in \mathbb{R}^e,$$ 
where the dimension of the resulting vector is equal to the 
number of 
all possible pairwise combinations of leaves in $T$. 
Hence the dimension of the phylogenetic tree space is 
$e = \binom{m}{2}$. 
In what follows, the connection between the space 
of phylogenetic trees and tropical linear spaces is established.

\begin{definition}[Ultrametric]
Consider the distance matrix
$D \in \mathbb{R}^{m \times m}$.
Then if 
\begin{eqnarray*}
\max\{D_{ij}, D_{jk}, D_{ik}\} 
\end{eqnarray*}
is attained at least twice for any $i,j,k \in [m]$, 
$D$ is an {\em ultrametric}.  
Note that the distance map $d(i,j) = D_{ij}$ forms a metric in 
$[m]$, with the strong triangular inequality satisfied.
The space of ultrametrics is denoted as 
$\Tn$.
\end{definition}

\begin{theorem}[noted in \cite{Buneman}]\label{thm:3pt}
Suppose we have an equidistant tree $T$ 
with a leaf label set $[m]$ and
$D$ as its distance matrix.
Then, $D$ is an ultrametric if and only if 
$T$ is an equidistant tree. 
\end{theorem}

Using Theorem \ref{thm:3pt}, if we wish to consider all possible equidistant trees, then it is equivalent to consider the space of ultrametrics as the space of phylogenetic trees on $[m]$.  Here we define $\Tn$ as the space of ultrametrics with a set of leaf labels $[m]$.
Theorem \ref{thm} (explained in \cite{AK,10.1093/bioinformatics/btaa564})
in Supplement \ref{sec:ult_thm} establishes the connection between 
phylogenetic trees and tropical geometry by stating that the ultrametric space is 
a tropical linear space.

\section{Method} \label{sec:method}
Our logistic regression model is designed to capture the association between a binary response variable 
$Y\in \{0,1\}$ and an explanatory variable vector $X\in\mathbb{R}^n$, 
where $n$ 
is the number of covariates in the model. 
Under the logistic model, 
$Y \sim \text{Bernoulli}(p(x|\omega))$ where
\begin{equation} \label{eq:logistic_model}
    p(x|\omega) = \mathbb{P}(Y=1|x) =
    \frac{1}{1+\exp{(-h_{\omega}(x)})} = 
    \sigma\left( h_{\omega}(x) \right),
\end{equation}
where $\sigma$ is the logistic function and 
$\omega \in\mathbb{R}^n$ is
the model parameter that needs to be estimated and $h$ 
is a function that will be specified later.
The log-likelihood function of logistic regression for $N$ observation pairs
$(x^{(1)},y^{(1)}), \dots, (x^{(N)},y^{(N)})$
is 
\begin{equation}
    \label{eq:log-like}
    l(\omega|x,y) = 
    \frac{1}{N} \sum_{i=1}^N y^{(i)} \log p_{\omega}^{(i)} + (1-y^{(i)}) 
    \log(1-p_{\omega}^{(i)}),
\end{equation}
where $p_{\omega}^{(i)} = p(x^{(i)}|\omega)$. 
It is the negative of the cross entropy loss.
The training model seeks a statistical estimator 
$\Hat{\omega}$ that maximizes this function. 

\subsection{Optimal Model}
The framework described thus far incorporates the tropical, classical and  BHV logistic regression as special cases. {\color{black} In this section, we show that }these can be distinguished through the choice of the function $h$. 
In fact, this function $h$ can be derived from the conditional distributions $X|Y$, 
as stated in Equation \eqref{eq:hu_gen} of 
Lemma \ref{lemma:p_bern}, below, by simple application
of the Bayes' rule.

If $X|Y$ is a Gaussian distribution 
with appropriate parameters, the resulting model is the
classical logistic regression. Alternatively, if $X|Y$ is a {\color{black}``}tropical{\color{black}''}  distribution,
then the resulting classification model is the {\color{black}``}tropical{\color{black}''} 
logistic regression. Examples \ref{ex:clas_log} and \ref{ex:trop_log} {\color{black} illustrate this for non-tropical and tropical distributions respectively, and Remark \ref{rm:tropLap} discusses the choice of tropical distribution in more detail.}

Furthermore, the function $h$ from \eqref{eq:hu_gen}
also minimizes the expected cross-entropy loss according to 
Proposition \ref{prop:p_functional}.
Therefore, the \textit{best model} to fit data that have been generated by 
tropical Laplace distribution \eqref{eq:trop_gauss} is the tropical 
logistic regression. {\color{black} We conclude this section showing how the tropical metric and tropical Laplace distribution may be applied to produce two intuitive variants of tropical logistic regression, our one- and two-species models.}
\vskip 0.2in



\begin{lemma}
    \label{lemma:p_bern}
    Let $Y \sim {\rm Bernoulli}(r)$
    and define the random vector $X \in \mathbb{R}^n$ 
    with conditional distribution
    $X|Y \sim f_Y$, where $f_0, f_1$ are probability density
    functions defined in $\mathbb{R}^n$. Then, 
    $Y | X \sim {\rm Bernoulli}(p(X))$ with 
    $p(x) = \sigma(h(x))$, where 
    \begin{equation}
        \label{eq:hu_gen}
        h(x) =
        \log \left(
            \frac{r f_1(x)}{(1-r)f_0(x)}
        \right).
    \end{equation}
\end{lemma}
\vskip 0.2in

\begin{proposition}
    \label{prop:p_functional}
    Let $Y \sim \text{Bernoulli}(r)$
    and define the random vector $X \in \mathbb{R}^n$ 
    with conditional distribution
    $X|Y \sim f_Y$, where $f_0, f_1$ are probability density
    functions defined in $\mathbb{R}^n$. 
    The functional $p$ that maximises the 
    expected log-likelihood as given by equation
    \eqref{eq:log-like} is $p(x) = \sigma(h(x))$, with 
    $h$ defined as in equation 
    \eqref{eq:hu_gen} of Lemma \ref{lemma:p_bern}.
\end{proposition}
\vskip 0.2in


\begin{example}[Normal distribution and classical logistic regression] \label{ex:clas_log}
Suppose that the two classes are equiprobable ($r=1/2$)
and that the covariate is multivariate normal 
\begin{equation*}
    X|Y \sim \mathcal{N}(\omega_Y,\sigma^2 I_n),
\end{equation*}
where $n$ is covariate dimension and $I_n$ is the identity
matrix.
Using Lemma \ref{lemma:p_bern}, 
the optimal model has 
\begin{equation}
    \label{eq:h_euclidean}
    h(x) = 
    - \frac{\|x-\omega_1\|^2}{2\sigma^2} + 
    \frac{ \|x-\omega_0\|^2}{2\sigma^2} = 
    \frac{(\omega_1-\omega_0)^T}{\sigma^2}
    (x-\Bar{\omega}), 
\end{equation}
where $\|\cdot \|$ is the Euclidean norm and $\Bar{\omega}=(\omega_0+\omega_1)/2$.
This model is the 
{\emph classical logistic regression} model 
with translated covariate $X-\Bar{\omega}$ and 
$\omega = \sigma^{-2} (\omega_1- \omega_0)$.
\end{example}
\vskip 0.2in

\begin{example}[Tropical Laplace distribution] \label{ex:trop_log}
It may be assumed that the covariates are distributed according to the
tropical version of the Laplace distribution, 
as presented in \cite{yoshida2022hit}, 
with mean $\omega_Y$
and probability density functions
\begin{equation}
    \label{eq:trop_gauss}
    f_Y(x) = \frac{1}{\Lambda} \exp \left(
        - \frac{d_{\rm tr}(x,\omega_Y)}{\sigma_Y}
    \right),
\end{equation}
where $\Lambda$ is the normalizing constant of the distribution.  \end{example}
\vskip 0.2in

\begin{proposition}
    \label{prop:normalizing_factor}
    In distribution \eqref{eq:trop_gauss}, 
    the normalizing factor is
    $\Lambda = e! \sigma_Y^{e-1}$.
\end{proposition}
\begin{proof}
    See Supplement \ref{sec:proofs}.
\end{proof}
\vskip 0.2in

{\color{black}
\begin{remark}\label{rm:tropLap}
    Consider $\mu\in \mathbb{R}^d$ and a covariance matrix $\Sigma \in \mathbb{R}^{d \times d}$.  Then the pdf of a classical Gaussian distribution is 
    \begin{equation}\label{eq:gausdist}
        f_{\mu, \Sigma}(x) \propto \exp\left(-\frac{1}{2}(x-\mu)^t \Sigma^{-1}(x-\mu)\right)
    \end{equation}
    where $x \in \mathbb{R}^d$ and $y^t$ is the transpose of a vector $y \in \mathbb{R}^d$.  When $\sigma_Y = 1$, the tropical Laplacian distribution in \eqref{eq:trop_gauss} is tropicalization of the left hand side in \eqref{eq:gausdist} where $\Sigma$ is to the tropical identity matrix
\[
\left(
\begin{array}{ccccc}
     0 &-\infty &-\infty & \ldots &-\infty  \\
     -\infty & 0 &-\infty & \ldots &-\infty  \\
     \vdots & \vdots & \vdots & \vdots &\vdots \\
     -\infty & -\infty &-\infty & \ldots &0 \\
\end{array}
\right) .
\]

Tran \cite{tran2020} nicely surveys the many different definitions of tropical Gaussian distributions. Since the space of ultrametrics is a tropical linear space \cite{SS}, it is natural to use tropical ``linear algebra'' for the definition of tropical ``Gaussian'' distribution defined in \eqref{eq:trop_gauss} in this research. Clearly not all desirable properties of the classical Gaussian distribution are necessarily realised in a tropical space.  

For example, as Tran discussed in \cite{tran2020}, we lose some natural intuition of orthogonality of vectors. This means that we lose a nice geometric intuition of a correlation between two random vectors.  Even with the loss of some nice properties of the classical Gaussian distribution, the tropical Laplacian \eqref{eq:gausdist} is a popular choice. It has been applied to statistical analysis of phylogenetic trees: as a kernel density estimator of phylogenetic trees over the space of phylogenetic trees \cite{YBMH}, and as the Bayes estimator \cite{10.1093/sysbio/syr021} because this distribution is interpretable in terms of phylogenetic trees.

In particular, the tropical metric $d_{\rm tr}$ represents the biggest difference of divergences (speciation time and mutation rates) between two species among two trees shown in Example \ref{eg:tropicalMetric}.  This is a very natural and desirable interpretation in terms of phylogenomics. The smaller difference of divergences between two species among the tree with an observed ultrametric $x$ and the tree with the centriod has higher probability.  Therefore, it is natural to apply a sample generated from the multi-species coalescent model where the species tree has the centroid as its dissimilarity map.  
It is worth noting that we do not know much about a well-defined distribution over the space of phylogenetic trees, despite many researchers' attempts \cite{Nye2021}.  
\end{remark}
}
\vskip 0.2in

{\color{black}
\begin{example}\label{eg:tropicalMetric}[Tropical Metric]
    \begin{figure}[h!]
        \centering
        \includegraphics[width=\textwidth]{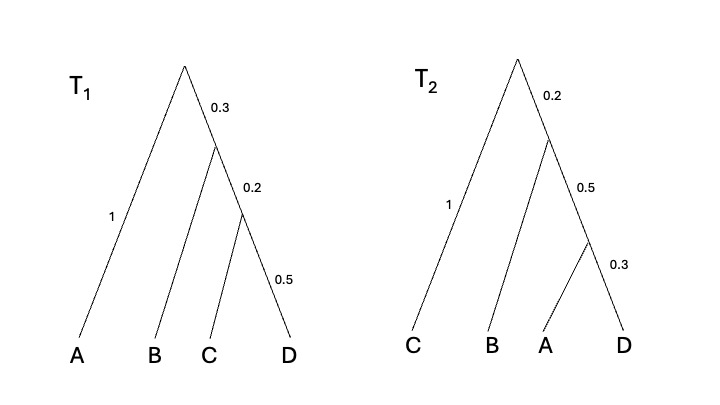}
        \caption{Example for an interpretation of the tropical metric $d_{\rm tr}$ in Example \ref{eg:tropicalMetric}.}
        \label{fig:tropicalmetric}
    \end{figure}
    Suppose we have equidistant trees $T_1$ and $T_2$ with leaf labels $\{A, B, C, D\}$ shown in Fig.~\ref{fig:tropicalmetric}.  Note that leaves $A$ and $C$ in $T_1$ and $T_2$ are switched.  Thus, the pairwise distances from $A$ and $D$ in $T_1$ and $T_2$, as well as he pairwise distances from $C$ and $D$ in $T_1$ and $T_2$ are the largest and second largest differences among all possible pairwise distances.
    
Let $u$ be a dissimilarity may from $T_1$ and $v$ be a dissimilarity map from $T_2$:
\[
\begin{array}{ccc}
     u &=&(2, 2, 2, 1.4, 1.4, 1)  \\
     v&=&(1.6,2,0.6,2,1.6,2). 
\end{array}
\]
Then we have
\[
u - v = (2-1.6, 2-2, 2-0.6, 1.4-2, 1.4-1, 1-2) = (0.4, 0,1.4,-0.6,0.4,-1).
\]
Therefore 
\[
d_{\rm tr}(u, v) = (u - v)_{A, D} - (u - v)_{C, D} 
\]
which means the tropical metric measures the difference of divergence between $A$ and $D$ and difference of divergence between $C$ and $D$.
\end{example}
\vskip 0.2in}

Combining the result of Proposition \ref{prop:normalizing_factor} with
Equations \eqref{eq:hu_gen} and \eqref{eq:trop_gauss} yields  
\begin{equation}
    \label{eq:hu_tropical}
    h_{\omega_0,\omega_1}(x) = 
    \frac{d_{\rm tr}(x,\omega_0)}{\sigma_0} - 
    \frac{d_{\rm tr}(x,\omega_1)}{\sigma_1} + 
    (e-1)
    \log\left(\frac{\sigma_0}{\sigma_1}\right).
\end{equation}
In its most general form, the model parameters are 
$(\omega_0,\omega_1,\sigma_0,\sigma_1)$ so the parameter space is a subset of
$(\trtor)^2 \times \mathbb{R}^2_+$ with dimension $2e$.
Two instances of this general model are particularly practically useful and interpretable. We call these the one-species and two-species models and they will be our focus for tropical logistic regression in the rest of the paper.

\par
For the \textit{one-species model}, 
it is assumed that $\omega_0 =\omega_1$ and 
$\sigma_0 \neq \sigma_1$. If, without loss of generality, 
$\sigma_1>\sigma_0$,
equation \eqref{eq:hu_tropical} becomes 
\begin{equation}
    \label{eq:h_one_species}
    h_{\omega}(x) = 
    \lambda \left( 
        d_{\rm tr}(x,\omega)  -  
        c
    \right), 
\end{equation}
where $\lambda = (\sigma_0^{-1} - \sigma_1^{-1})$ and
$\lambda c = \log{\left(\sigma_1/\sigma_0\right)}$.
Symbolically, 
the expression in equation \eqref{eq:h_one_species} can be considered to 
be a scaled tropical inner product, whose direct analogue in classical 
logistic regression is the classical inner product 
$h_{\omega}(x) = \omega^T x$. 
See Section \ref{Sec:tro:inner:prod} in the supplement for more details.
The classifier is $C(x) = \mathbb{I}(d_{\rm tr}(x,\Hat{\omega} ) > c) $, 
where $\Hat{\omega}$ is the inferred estimator of $\omega^*$. 
Note that the classification threshold and the probability contours ($p(x)$)
are tropical circles, illustrated in Figure \ref{fig:one_species_model}.
\par 
For the \textit{two-species-tree model}, it is assumed that 
$\sigma_0 = \sigma_1$, and $\omega_0\neq \omega_1$. 
Equation \eqref{eq:hu_tropical} reduces to 
\begin{equation}
    \label{eq:h_two_species}
    h_{\omega_0,\omega_1}(x) = 
    \sigma^{-1}(
    d_{\rm tr}(x,\omega_0) - 
    d_{\rm tr}(x,\omega_1)
    ),
\end{equation}
with a classifier
\(
    C(x) = \mathbb{I}(
        d_{\rm tr}(x,\Hat{\omega}_0) > 
        d_{\rm tr}(x,\Hat{\omega}_1)
    ),
\) where $\Hat{\omega}_y$ is the inferred tree for class $y \in \{0,1\}$.
The classification boundary is the tropical bisector
which is extensively studied in \cite{criado2021tropical} 
between the estimators $\Hat{\omega}_0$ and $\Hat{\omega}_1$ 
and the probability contours are tropical 
hyperbolae with $\Hat{\omega}_0$ and $\Hat{\omega}_1$ as foci,
as shown in Figure \ref{fig:two_species_model_euclidean}(right).
\par 
The one-species model 
is appropriate when the gene trees 
of both classes
are concentrated around the same species tree
$\omega$ with potentially different concentration rates.
When the gene trees of each class come from
distributions centered at different species 
trees the two-species model is preferred.

\subsection{Model selection} \label{sec:model_selection}

In the previous subsection, we established the correspondence between the 
covariate conditional distribution and the function $h$ which defines the logistic
regression model. According to Proposition \ref{prop:p_functional}, the best regression model follows from the distribution that fits the data. 
The family of distributions that best fits the training data of a given class can indicate
which regression model to use. 
The question that naturally arises is how to assess which family of 
conditional distributions has the best fit. 
\par
One issue is that the random covariates are multivariate and so the Kolmogorov–Smirnov test
can not be readily applied.
Moreover, the four families considered, namely the 
classical and tropical Laplace and Gaussian distributions, 
are not nested. Nonetheless, it is observed that for all these families 
the distances of the covariates from their centres are Gamma distributed.
This is stated in Corollary \ref{cor} which is based on Proposition \ref{prop:tropical_radius_dist}. Note that the distance metric corresponds to the geometry
of the covariates. However, the arguments used in the proof of Corollary \ref{cor}
do not work for distributions defined on the space of ultrametric trees
$\mathcal{U}_m$, because these spaces are not translation invariant.
For a similar reason, the corollary does not apply to the BHV metric. 
\vskip 0.2in

\begin{proposition}
    \label{prop:tropical_radius_dist}
    Consider a function 
    $d:\mathbb{R}^n \rightarrow \mathbb{R}$ 
    with $\alpha d(x) = d(\alpha x)$, for all 
    $\alpha\geq 0$.
    If $X \sim f$ with 
    $f(x) \propto \exp(-d^i(x)/(i\sigma^i))$ 
    being a valid probability
    density function, for some $i \in \mathbb{N},$ $ \sigma>0$. 
    Then, $d^i(X) \sim  i \sigma^i {\rm Gamma}(n/i)$.
\end{proposition}
\vskip 0.2in

\begin{corollary} \label{cor}
    If $X \in \mathbb{R}^e$ with $X\sim f \propto 
    \exp{(-d^i(x,\omega^*)/(i\sigma^{i}))}$, where $d$ is the Euclidean metric, then
    $d^i(X,\omega^*) \sim i \sigma^i {\rm Gamma}(e/i)$.
    If $X \in \trtor$ with $X\sim f \propto 
    \exp{(-d_{\rm tr}^i(x,\omega^*)/(i\sigma^{i}))}$, where $d_{\rm tr}$ is the tropical metric, then
    $d_{\rm tr}^i(X,\omega^*) \sim i \sigma^i {\rm Gamma}((e-1)/i)$.
\end{corollary}
\vskip 0.2in

The suitability of the tropical against the classical logistic regression is assessed 
for the coalescent model and the Mr Bayes trees, by visually comparing the fits of
the theoretical Gamma distributions to Euclidean and tropical distances of the gene trees
to the species tree.
\subsection{Consistency and Generalization Error}
In this subsection, the consistency of the statistical estimators 
(in Theorem \ref{prop:consistency_of_estimators}) and of 
the tropical logistic regression as a learning algorithm 
(in Propositions \ref{prop:gen_error_one_species} and 
\ref{prop:gen_error})
are established.
Finally, the generalization error (probability of misclassification for unseen data) for the one-species {\color{black}model} is derived and an 
upper bound is found for the generalization error of the two-species model.
In both cases the error bounds are getting better as the estimation error $\epsilon$ shrinks to zero.
It is worth mentioning that in the case of exact estimation, the generalization error of the 
one-species model can be computed explicitly by equation \eqref{eq:gen_error_one_species}.
Moreover, there is a higher misclassification rate from the more dispersed class 
(inequality \eqref{eq:outer_class_higher_misclassification}).
\vskip 0.2in

\begin{theorem}[Consistency]
    \label{prop:consistency_of_estimators}
    The estimator 
    $(\Hat{\omega},\Hat{\sigma}) = 
    (\Hat{\omega}_0, \Hat{\omega}_{1},
    \Hat{\sigma}_0,\Hat{\sigma}_1)
    \in \Omega^2 \times \Sigma^2$ 
    of the parameter 
    $
        (\omega^*,\sigma^*) = 
        (\omega_0^*,\omega_1^*,\sigma_0^*,\sigma_1^*)
        \in \Omega^2 \times \Sigma^2
    $
    is defined as the maximizer of the logistic likelihood function, 
    where 
    $\Omega \subset \mathbb{R}^e/\mathbb{R}{\bf 1}$
    and $\Sigma \subset \R_+$
    are compact sets.
    Moreover, it is assumed that the covariate-response pairs $(X_1,Y_1), (X_2,Y_2), \dots, (X_n,Y_n)$
    are independent and identically distributed with 
    $X_i \in \trtor$, $d_{\rm tr}(X,\omega_Y)$ 
    being integrable and square-integrable and
    $Y_i \sim \text{Bernoulli}(
    S(h(X_i,(\omega^*,\sigma^*) )))$.
    Then, 
    $$(\Hat{\omega},\Hat{\sigma}) 
    \overset{p}{\to} (\omega^*,\sigma^*)
    \text{ as } 
    n \to \infty.
    $$ 
    In other words, the model parameter estimator is 
    consistent.
\end{theorem}
\vskip 0.2in

\begin{proposition}[One-species generalization error]
    \label{prop:gen_error_one_species}
    Consider the one-species model where 
    $\omega=\omega_0=\omega_1 \in \trtor$ and 
    without loss of generality
    $\sigma_0 < \sigma_1$. 
    The classifier is 
    $C(x)=\mathbb{I}(h_{\Hat{\omega}}(x) \geq 0)$, 
    where $h$ is defined in equation \eqref{eq:h_one_species}
    and $\Hat{\omega}$ is the estimate for $\omega^{\star}$.
    Define the covariate-response joint random variable
    $(X,Y)$ with $Z = \sigma_Y^{-1} d_{\rm tr}(X,\omega_Y^*)$
    drawn from the same distribution with cumulative density 
    function $F$. Then, 
    \begin{align*}
        \mathbb{P}(C(X) = 1|Y=0) &\in \left[
            1-F(\sigma_1 \left( \alpha + \epsilon \right)),
            1-F(\sigma_1 \left( \alpha - \epsilon \right))  
        \right], \\
         \mathbb{P}(C(X) = 0|Y=1) &\in \left[
            F(\sigma_0 \left( \alpha - \epsilon \right)),
            F(\sigma_0 \left( \alpha + \epsilon \right))  
        \right], \text{ where } \\ 
        \alpha = \frac{\log{\frac{\sigma_1}{\sigma_0}}}{\sigma_1 - \sigma_0}, & 
        \,\,
        \text{ and }
        \,\,
        \epsilon = (e-1) \frac{d_{\rm tr}( \Hat{\omega}, \omega^* ) }{\sigma_1\sigma_0}.
    \end{align*}
    The generalization error defined as 
    $\mathbb{P}(C(X) \neq Y)$ 
    lies in the average of the two intervals above.
    In particular, note that if $\Hat{\omega} = \omega^*$, then $\epsilon=0$ and the intervals shrink to 
    a single point, 
    so the 
    misclassification probabilities and generalization error can be computed explicitly.
    \begin{equation}
        \label{eq:gen_error_one_species}
        \mathbb{P}(C(X) \neq Y) = \frac{1}{2} \left(
            1-F(\sigma_1 \alpha) + F(\sigma_0 \alpha)
        \right)
    \end{equation}
    Moreover, if $\Hat{\omega} = \omega_*$ and $Z \sim 
    {\rm Gamma}(e-1,1)$, then 
    \begin{equation}
        \label{eq:outer_class_higher_misclassification}
        \mathbb{P}(C(X)=1 | Y = 0) < \mathbb{P}(C(X)=0 | Y = 1).
    \end{equation}
\end{proposition}
\vskip 0.2in

\begin{proposition}[Two-species generalization error]
    \label{prop:gen_error}
    Consider the random vector 
    $X \in \mathbb{R}^e/\mathbb{R}{\bf 1}$
    with response 
    $Y \in \{0,1\}$ and the random variable
    $Z = d_{\rm  tr}( X , \omega^*_Y )$.
    Assuming that the probability density function is 
    $f_X(x) \propto f_Z(d_{\rm  tr}( x , \omega^*_Y ))$,
    the generalization error satisfies the 
    following upper bound
    \begin{equation}
        \label{eq:gen_error}
        \mathbb{P}\left( C(X) \neq Y \right) \leq 
        \frac{1}{2} 
            F^C_{Z}(\Delta_{\epsilon}) + h(\epsilon)
        ,
    \end{equation}
    where $\epsilon = d_{\rm  tr}(\Hat{\omega}_1 , \omega^*_1 )+d_{\rm  tr}(\Hat{\omega}_{0}, \omega^*_{0})$,
    $2\Delta_\epsilon =  \left(
    d_{\rm  tr}(\omega_1^*, \omega_{0}^*) - \epsilon \right)$,
    $F^C_Z$ is the complementary cumulative distribution of $Z$, 
    and 
    $h(\epsilon)$ is an increasing function of $\epsilon$ 
    with $2h(\epsilon) \leq F^C_Z(\Delta_{\epsilon})$ and
    $h(0)=0$ assuming that 
    $\mathbb{P}(d_{\rm  tr}(X,\omega_1^*)) = d_{\rm  tr}(X,\omega_{-1}^*)  
    )=0$.
    Moreover, under the conditions of Theorem 
    \ref{prop:consistency_of_estimators}, 
    our proposed learning algorithm is consistent.
\end{proposition}
\vskip 0.2in

Observe that the upper bound is a strictly increasing function of 
$\epsilon$. 
\vskip 0.2in

\begin{example}
    The complementary cumulative distribution of 
    ${\rm Gamma}(n,\sigma)$ is 
    $F^C(x) = \Gamma(n,x/\sigma)/\Gamma(n,0)$, 
    where $\Gamma$ is the 
    upper incomplete gamma function and
    $\Gamma(n,0)=\Gamma(n)$
    is the regular Gamma 
    function.
    Therefore, 
    the tropical distribution given in equation 
    \eqref{eq:trop_gauss} yields the 
    following upper bound
    for the
    generalization error 
    \begin{equation}
        \label{eq:upper_bound_gen_error}
        \frac{\Gamma\left(e-1, 
        \frac{d_{\rm tr}(\omega_0^*,\omega_1^*)}{2\sigma}
        \right)}
        {2\Gamma(e-1)},
    \end{equation}
    under the assumptions of Proposition \ref{prop:gen_error} 
    and assuming that the estimators coincide with the 
    theoretical parameters. This assumption is reasonable for large sample sizes and it follows from 
    Theorem \ref{prop:consistency_of_estimators}.
\end{example}
\vskip 0.2in

In subsequent sections, these theoretical results will guide us in implementing our model.
Bounds on the generalization error from Propositions
\ref{prop:gen_error_one_species} and
\ref{prop:gen_error} are computed and 
the suitability of
Euclidean and tropical distributions, 
and as a result of classical and tropical logistic regards,
is assessed
using the distance distribution 
of Proposition \ref{prop:tropical_radius_dist}.  

\section{Optimization}  \label{sec:optimization}
As in the classical logistic regression, the parameter vectors $(\hat\omega, \hat\sigma)$ maximising the log-likelihood \eqref{eq:log-like}, are chosen as statistical estimators. Identifying these requires the implementation of a continuous optimization routine. While root-finding algorithms typically work well for identifying maximum likelihood estimators in the classical logistic regression where the log-likelihood is concave, they are unsuitable here. The gradients of the log-likelihood under the proposed tropical logistic models are only piecewise continuous, with the number of discontinuities increasing along with the sample size. Furthermore, even if a parameter is found, it may merely be a local optimum. In light of this, the tropical Fermat-Weber problem of \cite{lin2018tropical} is revisited.

\subsection{Fermat-Weber Point}
A Fermat-Weber point or geometric mean $\Tilde{\omega}_n$ of the sample set $(X_1,\dots,X_n)$ 
is a point that minimizes the sum of distances from to sample points, i.e.
\begin{equation}
    \label{eq:fermat_weber_point}
    \Tilde{\omega}_n \in \argmin_{\omega} \sum_{i=1}^n d_{\rm tr}(X_i,\omega).
\end{equation}
This point is rarely unique for finite $n$, {\color{black} indeed there will often be an infinite set of Fermat-Weber points \cite{lin2018tropical}}. However, the proposition below gives conditions for asymptotic convergence. 
\vskip 0.2in

\begin{proposition}
    \label{prop:consistency_of_Fermat_Weber_points}
    Let $X_i \overset{\rm iid}{\sim} f$, where 
    where $f$ is 
    a distribution that is symmetric around its center
    $\omega^* \in \mathbb{R}^{e}/\mathbb{R}{\bf 1}$ i.e. 
    $f(\omega^* + \delta) = f(\omega^* - \delta)$ for all 
    $\delta \in \mathbb{R}^{e}/\mathbb{R}{\bf 1}$. 
    Let $\Tilde{\omega}_n$ be any Fermat-Weber point as
    defined in equation \eqref{eq:fermat_weber_point}. 
    Then, $\Tilde{\omega}_n \overset{p}{\to} \omega^*$ 
    as $n \to \infty$.
\end{proposition}
\vskip 0.2in

The significance of Proposition \ref{prop:consistency_of_Fermat_Weber_points} is twofold.
It proves that the Fermat-Weber sets of points sampled from symmetric distributions
tend to a unique point. This is a novel result and ensures that for 
sufficiently large sample sizes the topology of any Fermat-Weber point is fixed.
Additionally, using Theorem \ref{prop:consistency_of_estimators} and 
Proposition \ref{prop:consistency_of_Fermat_Weber_points}, 
$\Hat{\omega}_n - \Tilde{\omega}_n  \overset{p}{\to}  0$ 
as $n \to \infty$. 
Furthermore, empirical evidence in Figure \ref{fig:asymptotic_error}, see the following section, suggests that 
$d_{\rm tr}(\Hat{\omega}_n, \omega^*) = \mathcal{O}_p(1/\sqrt{n})$
and $d_{\rm tr}(\Tilde{\omega}_n, \omega^*) = \mathcal{O}_p(1/\sqrt{n})$. 
These statements are left as conjectures and proofs of them are beyond the scope of this paper.
Assuming they hold and applying triangular inequality, it follows that 
$ d_{\rm tr}(\Hat{\omega}_n, \Tilde{\omega}_n)= \mathcal{O}_p(1/\sqrt{n}).$
As a result, for a sufficiently large sample size we may use 
the Fermat-Weber point as an approximation for the MLE vector.
Indeed, there are benefits in doing so.\\ 
\par 
Instead of having a single optimization problem with $2e-1$ variables, 
three simpler problems are considered; 
finding the Fermat-Weber point of each of the two classes, which has $e-1$ degrees 
of freedom and then finding the optimal $\sigma$ which is a one dimensional 
root finding problem. The algorithms of our implementation for both model
can be found in Supplement \ref{sec:alg}.

There is also another another benefit of using Fermat-Weber points.
Proposition \ref{prop:vanishing_gradient} provides a sufficient optimality condition
that the MLE lacks, since a vanishing gradient in the log likelihood function 
merely shows that there is a local optimum.
\vskip 0.2in

\begin{proposition}
    \label{prop:vanishing_gradient}
    Let $X_1,\dots,X_n \in \trtor$, 
    $\omega \in \trtor$ and define the function
    \begin{equation*}
        f(\omega) = \sum_{i=1}^{n} d_{\rm tr}(X_i,\omega).
    \end{equation*}
    \begin{enumerate}[i.]
        \item The gradient vector of $f$ is defined at $\omega$ if and only if
        the vectors $\omega - X_i$ have unique maximum and minimum components 
        for all $i \in [n]$. 
        \item If the gradient of $f$ at $\omega$ is well-defined and zero, then 
        $\omega$ is a Fermat-Weber point.
    \end{enumerate}
\end{proposition}

In \cite{lin2018tropical}, Fermat-Weber points are computed 
by means of linear programming, which is computationally 
expensive. Employing a gradient-based method is much faster,
but there is no guarantee of convergence. 
Nevertheless, if the gradient, which is an integer vector, 
vanishes, then it is guaranteed, as above, that the algorithm 
has reached a Fermat-Weber point. This tends to happen rather frequently, but not in all cases examined in Section
\ref{sec:results}.
\vskip 0.2in
{\color{black}
\begin{remark}
    Our choice of Fermat-Weber points to represent centers is not the only practical option, however it is an especially desirable choice due to the interpretability of its resulting solutions. 
    
    Recently, Com\v aneci and Joswig studied  tropical Fermat-Weber points obtained using the asymmetric tropical distance  \cite{joswig2023}. They found that if all $X_i$ are ultrametric, then the resulting tropical Fermat-Weber points are also ultrametric, all with the same tree topology.  
    On the other hand, Lin et al. \cite{LSTY} show that a tropical Fermat-Weber point defined with $d_{\rm tr}$ of a sample taken from the space of ultrametrics could fall outside of the ultrametric space.  
    
    Despite this, the major drawback of using the asymmetric tropical distance, is that it would result in losing the phylogenetic interpretation of the distance or dissimilarity between two trees  held by the tropical metric $d_{\rm tr}$ - see Remark \ref{rm:tropLap}.  
\end{remark}
}
\section{Results} \label{sec:results}
In this section, tropical logistic regression is applied 
in three different scenarios. 
The first and simplest 
considers datapoints generated from the tropical Laplace distribution. 
Secondly, gene trees sampled from a coalescent model are classified based on the species tree they have been generated from, and finally 
it is applied as an MCMC convergence criterion for the 
phylogenetic tree construction, using output from the
\texttt{Mr Bayes} software.
The models' performance in terms of misclassification rates
and AUCs on these datasets 
is examined.

\subsection{Toy Example}
In this example, a set of data points is generated from the 
tropical normal distribution 
as defined in Equation \eqref{eq:trop_gauss} 
using rejection sampling. 

\begin{figure}[h]
    \centering
    \includegraphics[width=.6\linewidth]{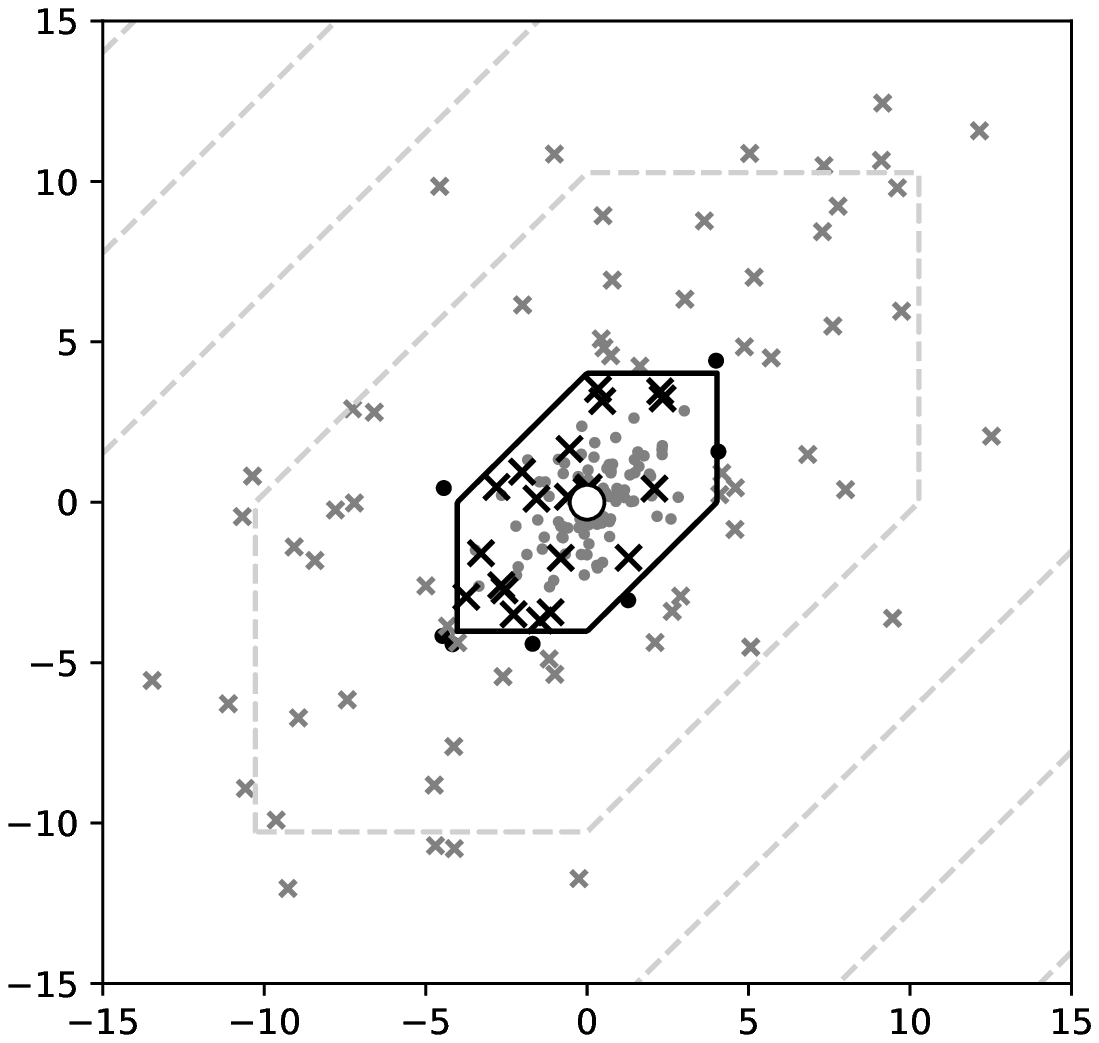}
    \caption{Scatterplot of $200$ points - $100$ dots for class $0$ 
    and $100$ Xs for
    class $1$, black for misclassified and grey otherwise - 
    imposed upon a contour plot of the probability of inclusion in class 0, where the 
    black contour is the classification threshold.
    The deviation parameters used in data generation 
    were $\sigma_0=1,\sigma_1=5$ and the centre of the distribution 
    (white-filled point) is the origin. 
    The centres of the two distributions are 
    $\omega_0 = \omega_1$.
    }
    \label{fig:one_species_model}
\end{figure}
\par 
The data points are defined in the tropical projective 
torus $\trtor$, which is isomorphic to $\R^{e-1}$. 
To map $x \in \trtor$ to $\R^{e-1}$, simply set the last component 
of $x$ to $0$, or in other words 
$x \mapsto (x_1-x_e,x_2-x_e, \dots, x_{e-1} - x_e)$.
For illustration purposes, it is desirable to plot points in 
$\mathbb{R}^2$, so we use $e=3$ which corresponds to phylogenetic trees 
with $3$ leaves. 
Both the one-species model and the two-species model are examined. 
\begin{figure}[h]
    \centering
    \includegraphics[width=.49\linewidth]{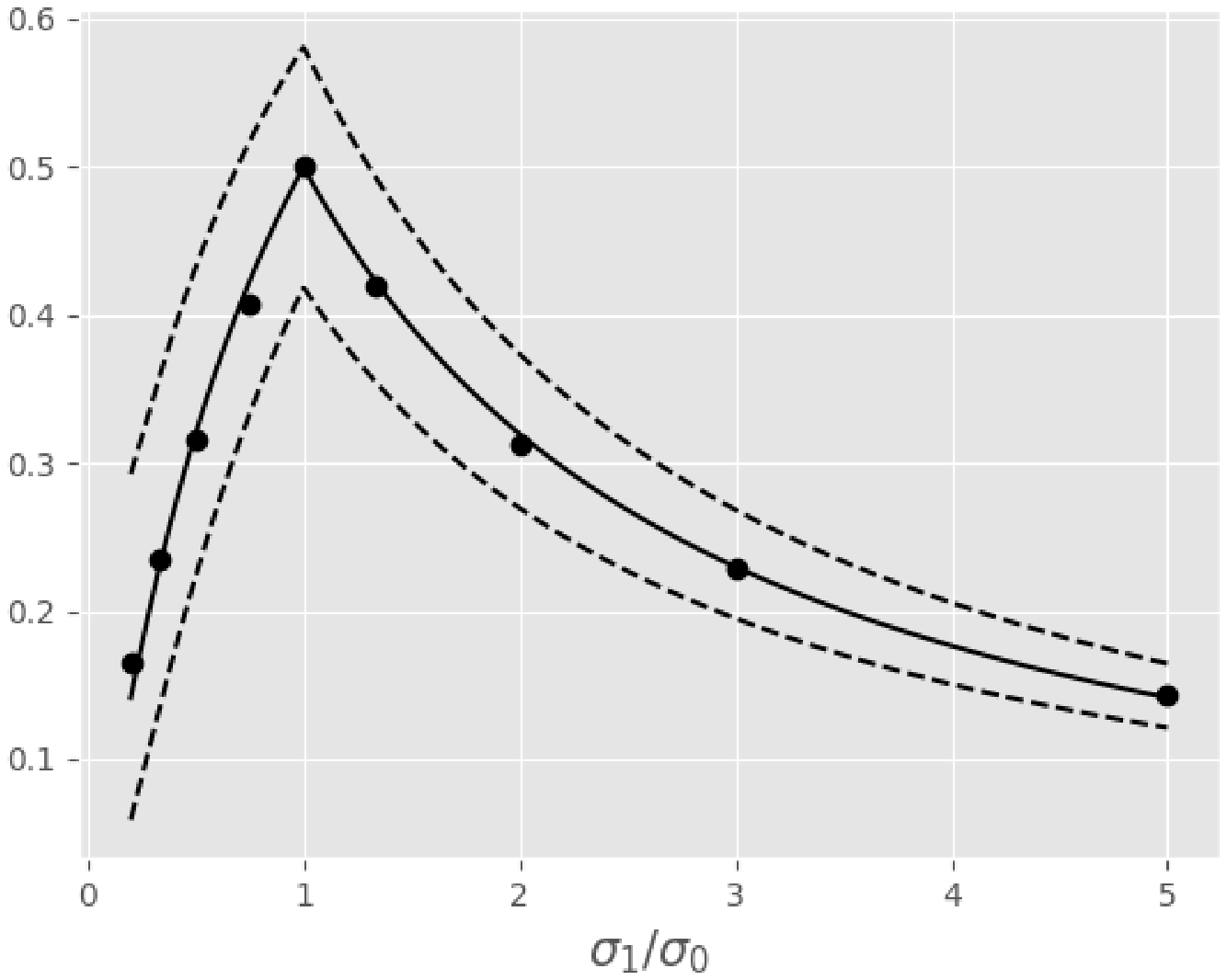}
    \includegraphics[width=.49\linewidth]{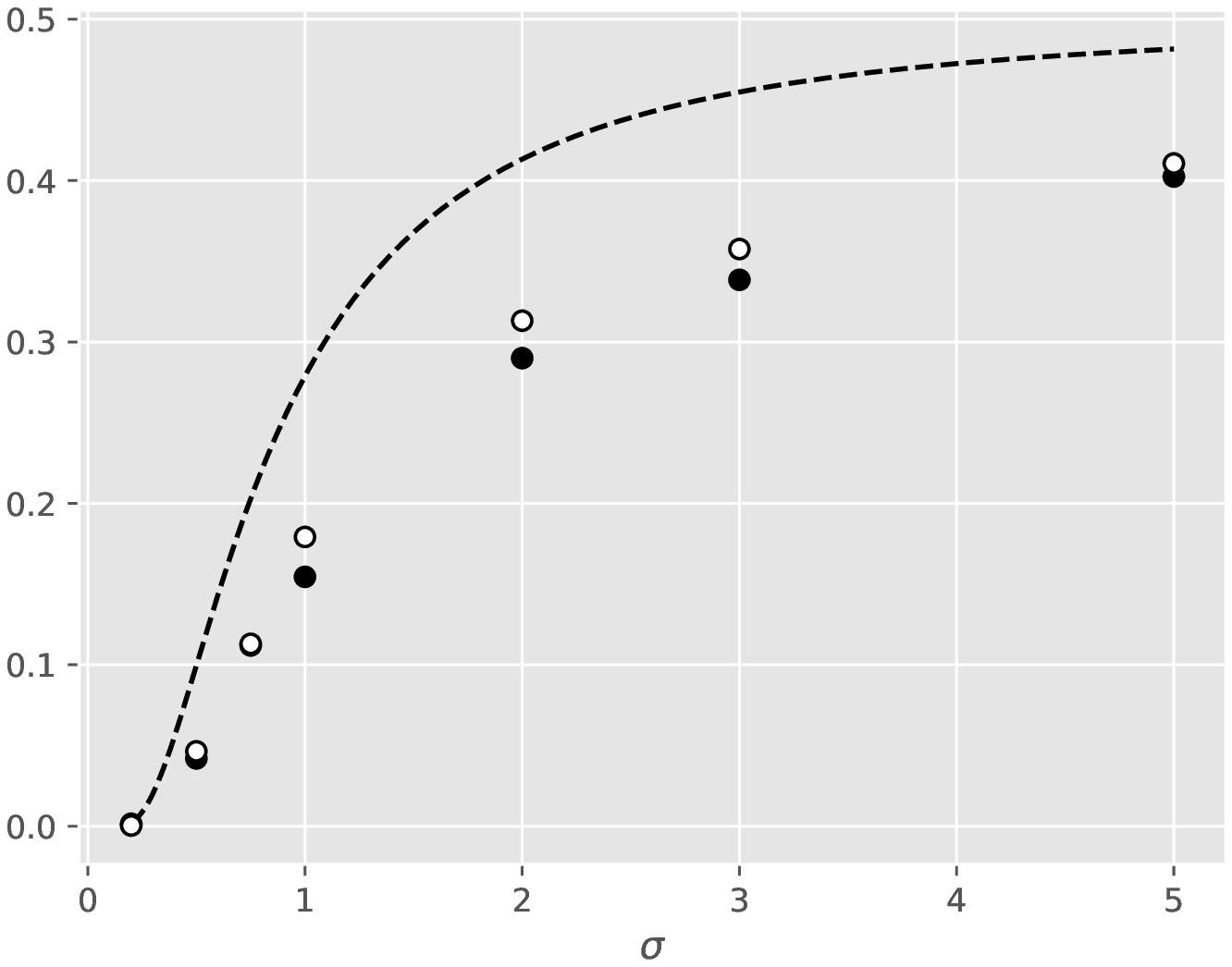}
    \caption{(left) Generalization error for $9$ different 
    deviation ratios. The estimator 
    $\Hat{\omega} = (0.3,0,3)$  differs from the true 
    parameter $\omega = (0,0)$. The upper and lower bounds
    of Proposition \ref{prop:gen_error_one_species} are
    plotted in dashed lines and the generalization error 
    for the correct estimator $\Hat{\omega}= \omega^*$ plotted
    in solid line. The dots represent the proportion of 
    misclassified points from a set of $2000$ points in 
    each experiment, $1000$ points for each class.
    (right) Generalization errors for 7 different dispersion parameters
    with black markers 
    for the two-species tropical logistic regression and white markers
    for the classical logistic regression.
    The upper bound \eqref{eq:upper_bound_gen_error}
    of Proposition \ref{prop:gen_error} is plotted in dashed line. 
    }
    \label{fig:gen_error}
\end{figure}
\par 
In the case of the former, 
$\omega = \omega_0 = \omega_1$ and $\sigma_0 \neq \sigma_1$. 
The classification boundary in this case is a tropical circle. 
If $\sigma_0 < \sigma_1$, the algorithm classifies points close 
to the inferred centre to class $0$ and those that are more 
dispersed away from the centre as class $1$. 
For simplicity, the centre is set to be the origin 
$\omega=(0,0,0)$
and no inference is performed. 
In Figure ~\ref{fig:one_species_model} a scatterplot of 
of the two classes is shown, where misclassified points are 
highlighted. 
As anticipated from Proposition 
\ref{prop:gen_error_one_species} 
there are more misclassified points 
from the more dispersed class (class $1$). 
Out of $100$ points for each class,
there are $7$ and $21$ misclassified points 
from class $0$ and $1$ 
respectively, while the theoretical probabilities 
calculated from equation \eqref{eq:gen_error_one_species}
of Proposition 
\ref{prop:gen_error_one_species} are 
$9\%$ and $19\%$ respectively. 
\begin{figure}[h]
    \centering
    \includegraphics[width=.49\linewidth]{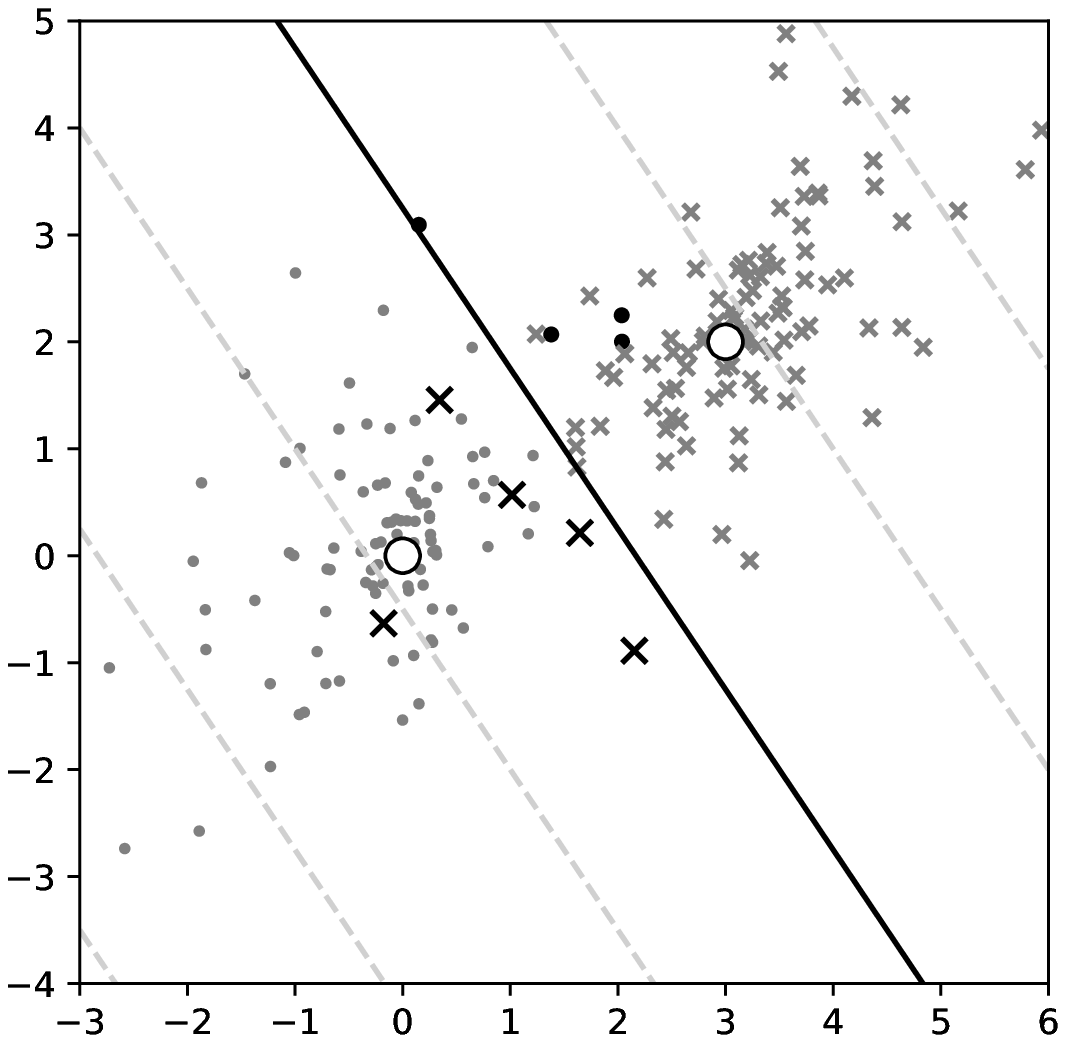}
    \includegraphics[width=.49\linewidth]{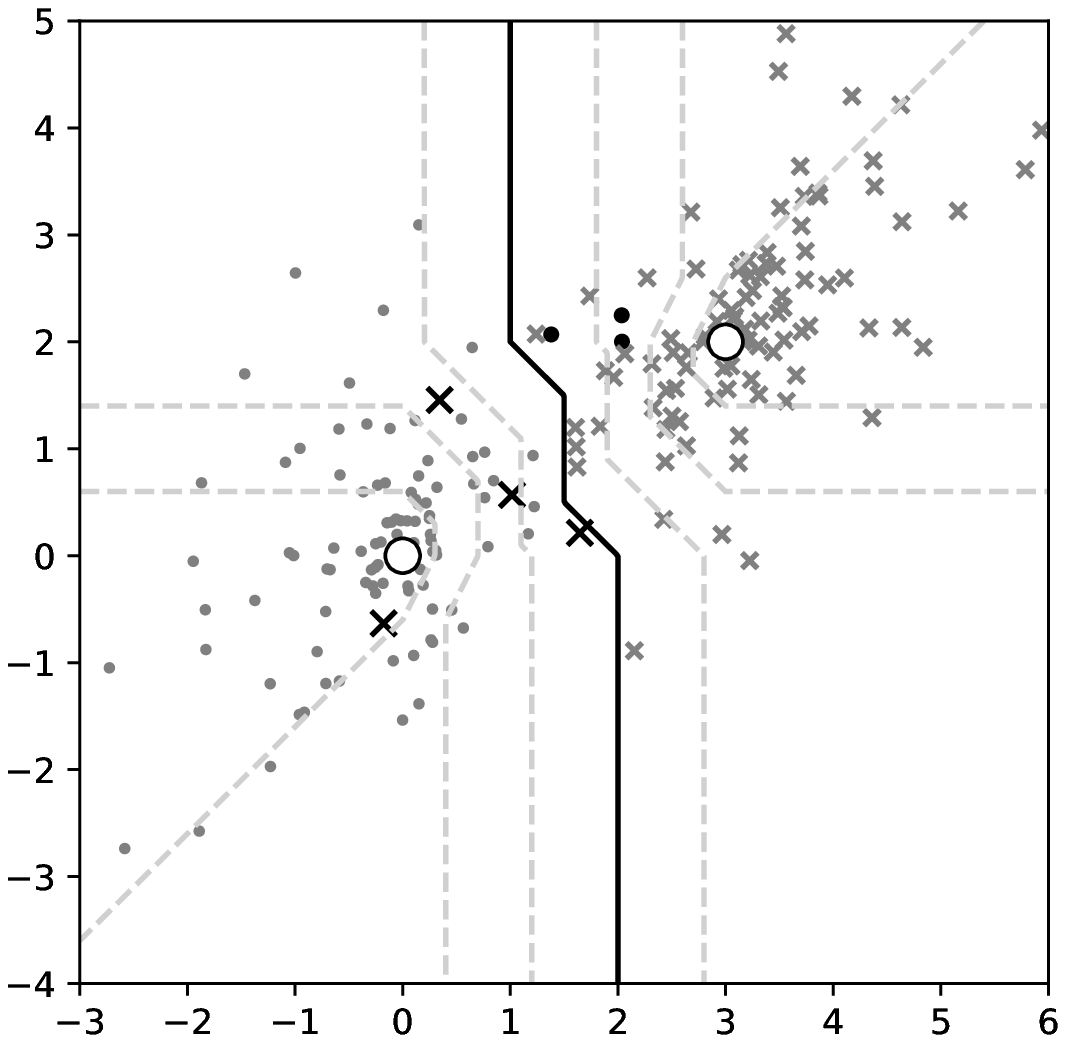}
    \caption{
    Scatterplot of points - dots for class $0$ and X for
    class $1$, black for misclassified according to 
    (left) {\bf classical logistic regression} or (right) {\bf tropical logistic regression},
    and grey otherwise - 
    alongside a contour plot of the probabilities, where the 
    black contour is the classification threshold.
    The centres, drawn as big white dots, 
    are $\omega_0 = (0,0,0), 
    \omega_1=(3,2,0)$ and $\sigma = 0.5$.}
    \label{fig:two_species_model_euclidean}
\end{figure}
\par
Varying the deviation ratio $\sigma_1/\sigma_0$ in the
data generation process allows exploration of its
effect on 
the generalization error in the one-species model.
The closer this ratio is to unity, the higher the 
generalization error. For $\sigma_0 = \sigma_1$ 
the classes are indistinguishable and hence any model
is as good as a random guess i.e. the generalization 
error is $1/2$. The estimate of the 
generalization error for every value of that ratio
is the proportion of misclassified points in both classes.
Assuming an inferred $\omega$ 
that differs from the true parameter, 
Fig.~\ref{fig:gen_error}(left)  
verifies the bounds of Proposition 
\ref{prop:gen_error_one_species}. 
\par 
For the two-species model, tropical logistic 
regression is directly compared to classical 
logistic regression. Data is generated using 
different centres 
$\omega_0 = (0,0,0),$ $\omega_1 = (3,2,0)$ but the same 
$\sigma=0.5$. 
The classifier is $C(x)=\mathbb{I}(h(x)>0)$ for both
methods, using $h$ as defined in equations
\eqref{eq:h_euclidean} and \eqref{eq:h_two_species}
for the classical and tropical logistic regression 
respectively.
Fig.~\ref{fig:two_species_model_euclidean} compares 
contours and classification thresholds of the classical (left) and tropical (right)
logistic regression by overlaying them on top of the same data.
Out of $100+100$ points 
there are $5+4$ and $4+3$ misclassifications 
in classical and tropical logistic regression respectively. 
Fig.~\ref{fig:gen_error}(right)
visualizes the misclassification 
rates of the two logistic regression methods for different values 
of dispersion $\sigma$, 
showing the tropical logistic regression to have consistently lower 
generalization error than the classical, even in this simple toy problem. \\
\par 
Finally, we investigate the convergence rate of the Fermat-Weber points and 
of the MLEs from the two-species model as the sample size $N$ increases. 
Fixing $\omega_0^* = (0,0,0)$ and $\omega_1^* = (3,2,0)$ as before, 
the Fermat-Weber point numerical solver and the log-likelihood optimization solver are 
employed to find $(\Tilde{\omega}_0)_N$ and 
$( (\Hat{\omega}_0)_N, (\Hat{\omega}_1)_N, \Hat{\lambda}_N )$ respectively.
From this, the error is computed for the two methods, which is defined as
$d_N = d_{\rm tr}((\omega_0)_N,\omega_0^*) $ for $(\omega_0)_N= (\Tilde{\omega}_0)_N$ and $(\Hat{\omega}_0)_N$ respectively. For each $N$, we repeat this procedure $100$ times to get an 
estimate of the mean error rate $r_N = 
\mathbb{E}\left( d_N \right)$.
Figure \ref{fig:asymptotic_error} shows that for both methods, $r_N \sqrt{N} \to C$ as $N \to \infty$, with $C_{\rm FW} < C_{\rm MLE}$. Since $\mathbb{E}(\sqrt{N} d_N) \to C$, it follows that
$\sqrt{N} d_N = \mathcal{O}_p(1)$ as $N \to \infty$. 
This supports the assumption of Section \ref{sec:optimization}
that Fermat-Weber points can be used in lieu of MLEs, since they converge to each other in probability at rate $1/\sqrt{N}$.
Interestingly, the MLEs produce higher errors than FW points. This may be due to an imperfection of the MLE solver, which may be stuck at a local optimum.

\begin{figure}
    \centering
    \includegraphics[width=.7\linewidth]{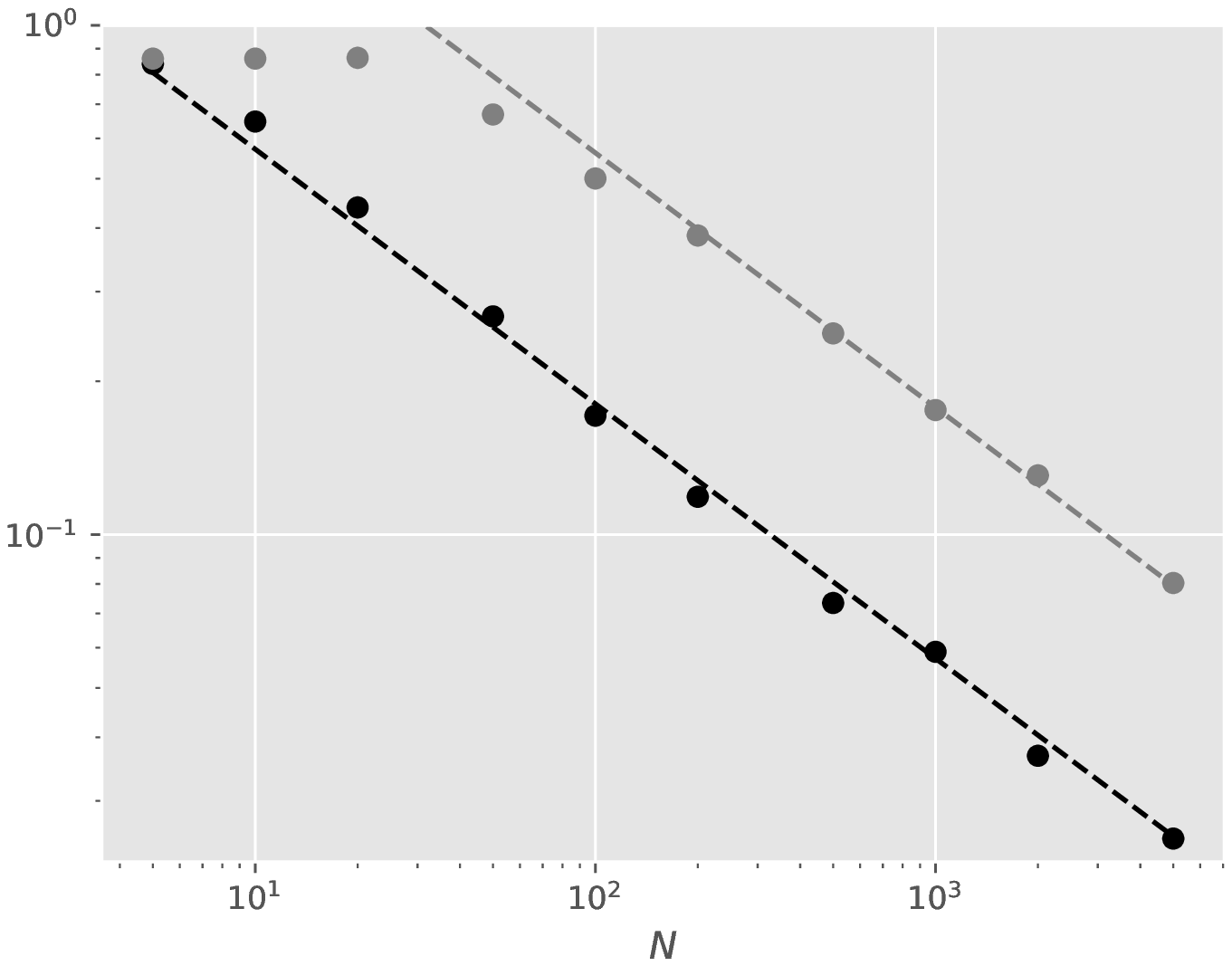}
    \caption{Expected asymptotic error for FW points $(\Tilde{\omega}_0)_N$ (in black) and MLE points 
    $(\Hat{\omega}_0)_N$ (in grey) for different values of $N$. 
    Error is defined as the tropical distance
    from the true centre $\omega_0^*$ i.e. $d_{\rm tr}(\omega_N,\omega_0^*)$.
    The dashed lines are $y \propto N^{-0.5}$, so this figure illustrates that 
    $d_{\rm tr}((\omega_0)_N,\omega_0^*) = \mathcal{O}_p(1/\sqrt{N})$ as $N\to \infty$.}
    \label{fig:asymptotic_error}
\end{figure}
\subsection{Coalescent Model} \label{sec:coalescent_model}
The
data that have been used in our simulations were generated 
under the multispecies coalescent model, 
using the python library \texttt{dendropy} 
\cite{sukumaran2010dendropy}.
The classification method we propose is the two-species model because two distinct 
species tree have been used to generate gene tree data for each class.
\par 
Two distinct species trees are used, which {\color{black}were} randomly generated  
under a Yule model. 
Then, using \texttt{dendropy}, $1000$ gene trees are randomly generated 
for each of the two species. 
The trees have $10$ leaves and so the number of the model variables is $\binom{10}{2} = 45$. 
They are labelled according to the species tree they are generated from.
The tree generation is
under the coalescent
model for specific model parameters. 
\par 
Since the species trees are known, we conduct a comparative 
analysis between classical, tropical and a
BHV-based (\cite{BHV})
logistic regression.
In the supplement, we show an approximation
analog of our model to the BHV metric.
The comparative analysis
includes the distribution fitting
of distances and
the misclassification rates
for different metrics.
\par 
In Fig.~\ref{fig:hist_dists},
the distribution
of the radius $d(X,\omega)$
as given by
Proposition \ref{prop:tropical_radius_dist},
is fitted to the histograms 
of the Euclidean and tropical distances 
of gene trees to their corresponding 
species tree, along with the corresponding pp-plots
on the right. 
According to Proposition \ref{prop:tropical_radius_dist},
for both the classical and tropical 
Laplace distributed covariates,
$d(X,\omega^*) \sim \sigma {\rm Gamma}(n)$, shown in solid lines 
in Fig.~\ref{fig:hist_dists}, where $n = e = 45$ 
and $n=e-1=44$ for the classical 
and tropical case respectively. 
Similarly, for normally distributed covariates,
$d(X,\omega^*) \sim \sigma \sqrt{\chi_{n}^2}$, shown in dashed lines.
It is clear that Laplacian distributions produce better fits in both 
geometries and that the tropical Laplacian fits the data best. 
As discussed in Section \ref{sec:model_selection}, 
the same analysis can not be applied to the BHV metric, 
because the condition of
Proposition \ref{prop:tropical_radius_dist} does not hold.
\par  
\textit{Species depth} $\text{SD}$
is the time since the speciation event between the species and
\textit{effective population size} $N$ quantifies 
genetic variation in the species. 
Datasets have been generated for a range of values
\(R := {\rm SD}/N \)
by varying species depth. 
For low values of $R$, 
speciation happens very recently and so the gene trees look 
very much alike. Hence, classification is hard for datasets with low values of
$R$ and vice versa, 
because the gene deviation $\sigma_R$ is a decreasing function of $R$.
We expect classification to improve in line with $R$.
Fig.~\ref{fig:rocs} and Fig.~\ref{fig:dist_hists} in Supplement \ref{sec:coalescent_graphs}  
confirm that, by showing that as $R$ increases
the receiver operating characteristic (ROC) curves
are improving and the Robinson-Foulds and tropical distances
of inferred (Fermat-Weber point)
trees are decreasing.
In addition, Fig.~\ref{fig:aucs} shows that as $R$ increases,
AUCs increase (left) and misclassification rates decrease (right).
It also shows that tropical logistic regression produces higher AUCs than classical 
logistic regression 
{\color{black} and other out-of-the-box ML classifiers such as random forest classifier, 
neural networks with a single sigmoid output layer and support vector machines. Our model also produces} 
 lower misclassification rates than both the BHV and 
classical logistic regression.
Finally, note that the generalization error upper bound as given in 
equation \eqref{eq:upper_bound_gen_error}
is satisfied but it not very tight 
(dashed line in Fig.~\ref{fig:aucs}). 

\begin{figure}
    \centering
    \includegraphics[width=0.9\textwidth]{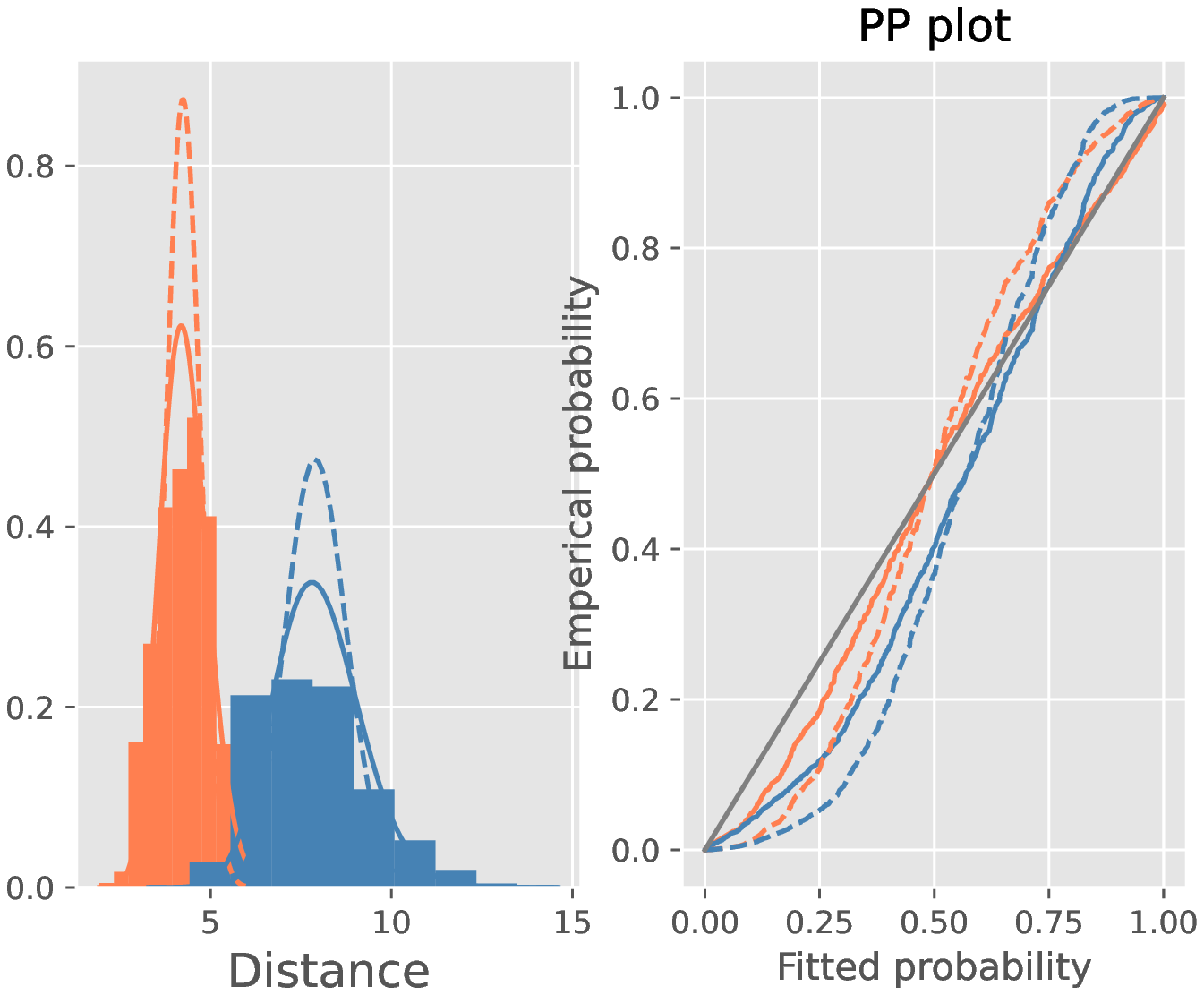}
    \caption{(Left) Histograms of the distances 
    of $1000$ 
    gene trees from the species trees that 
    generated them under the coalescent model
    with $R=0.7$.
    Coral and blue corresponds to tropical and euclidean 
    geometries respectively.
    The solid and dashed lines are 
    fitted distributions 
    $\sigma {\rm Gamma}(n)$  and 
    $\sigma \sqrt{\chi^2_{n}}$
    respectively; $\sigma$ 
    is chosen to be 
    the MLE, derived in the supplement.
    Euclidean metric has worse 
    fit than  
    the tropical metric.
    This can also be observed by the corresponding pp-plots (right). 
    }
    \label{fig:hist_dists}
\end{figure}

\begin{figure}[h]
    \centering
    \includegraphics[width=0.45\textwidth]{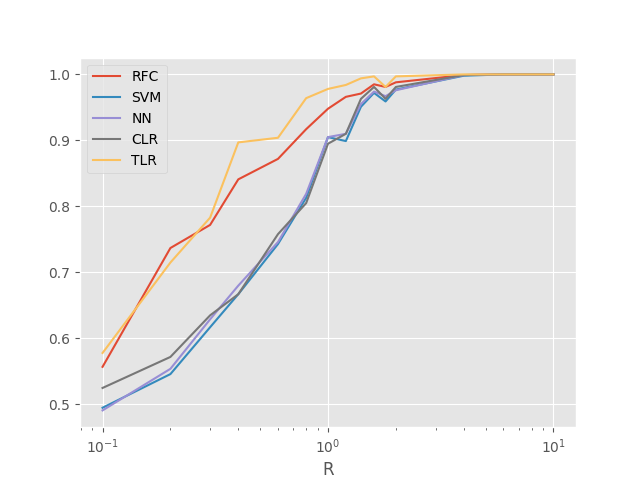}
    \includegraphics[width=.45\textwidth]{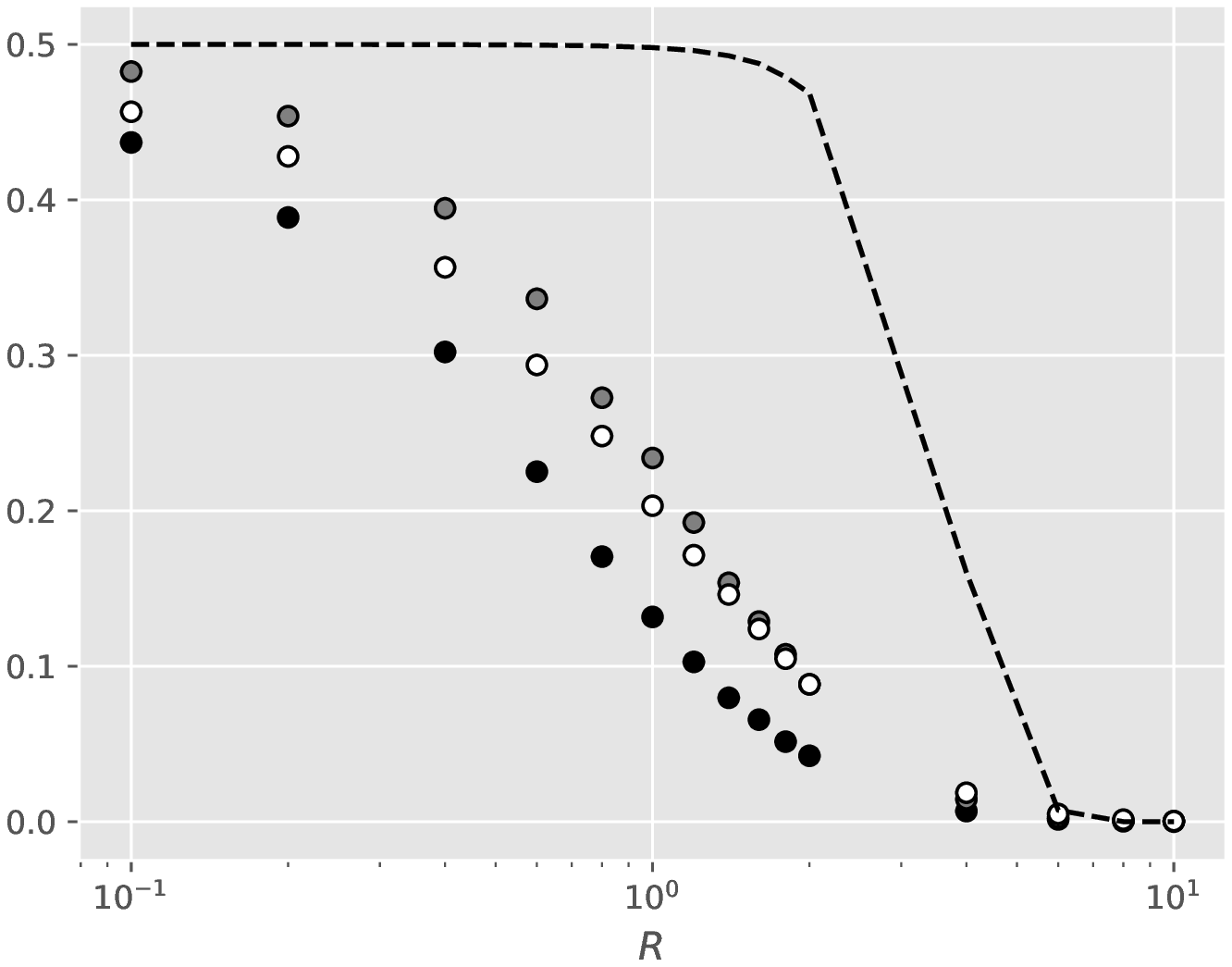}
    \caption{(left) 
    {\color{black} 
        Average AUCs against $R$. Five classification models which we considered are the tropical two species-tree model (TLR), random forest classifier (RFC), support vector machines (SVM), neural networks (NN) and classical logistic regression (CLR). We used default set up for TLR, SVM, NN and CLR implemented by \texttt{sklearn}.
    }
    (right) the x-axis represents the ratio $R$ and the y-axis represents misclassification rates. Black circles represent the tropical logistic regression, white circles represent the classical logistic regression,  grey points represent the logistic regression with BHV metric, and the dashed line represents the theoretical generalization error shown in Proposition \ref{prop:gen_error}.}
    \label{fig:aucs}
\end{figure}

\subsection{Convergence of Mr Bayes}\label{sec:mb}
\texttt{Mr Bayes} (\cite{huelsenbeck2001bayesian})
is a widely used software for Bayesian inference of 
phylogeny using MCMC to sample the target posterior
distribution. 
An important feature of the software is the diagnostic 
metrics indicating whether a chain has converged 
to the equilibrium distribution.
This is calculated at regular, specified intervals, set
by the variable \texttt{diagnfreq}, using the average 
standard deviation of split frequencies 
(ASDSF introduced by \cite{lakner2008efficiency})
between two independently run chains. 
The more similar the split frequencies between the 
two chains are, the lower the ASDSF, and the more likely 
it is that both chains have reached
the equilibrium distribution. \\
Our classification model provides an alternative convergence 
criterion for MCMC convergence. 
Consider two independently run chains; the sampled trees of the 
two chains correspond to two classes and the AUC value is a 
measure of how distinguishable the two chains are.  
High values of AUC are associated with easily distinguishable 
chains, implying that the chains have not converged to the 
equilibrium distribution.
At every iteration that is a multiple of \texttt{diagnfreq}, the ASDSF
metric is calculated and the AUC of the two chains is found 
by applying tropical logistic regression to the truncated chains
that only keep the last $30 \%$ of the trees in each chain.
\\ 
For our comparison study, the data used were the gene
sequences from the \texttt{primates.nex} file. 
This dataset comes with the \texttt{Mr Bayes} software 
and it is used as an example in 
\cite{ronquist2005mrbayes}.
Figure \ref{fig:primates} shows the two metrics at different iterations of the two independent chains ran on this dataset. According to the \texttt{Mr Bayes} manual, the convergence threshold for their metric is $10^{-2}$. This is achieved at the $800$-th iteration, when our method
produces an AUC of $97\%$, which indicates that the chains may have not converged yet, 
contrary to the suggestion of Mr Bayes. 
A likely explanation for this discrepancy is the
dependence of ASDSF on tree topologies instead of branch lengths. The frequencies of the tree topologies may have converged to those of the equilibrium distribution, even if the branch lengths have not. 
Eventually, the AUC values drop rapidly when iterations exceed $2\cdot 10^3$, while the ASDSF metric is reduced at a much slower rate. In this second phase, the branch lengths are calibrated, while the topology frequencies do not change a lot. 
Finally, for iterations that exceed  $10^5$, neither metric can reject convergence, with 
ASDSF being $10$ lower than the threshold and the AUC values finally dropping below $70\%$,
which is a typical threshold for poor classification.\\ 
{\color{black} When our classification method is compared to other classifiers, it marginally outperforms classical logistic regression and neural networks with a single sigmoid output but underperforms support vector machines and random forest classifiers. 
Despite their simplicity, logistic regression models cannot capture the complexity of the
chain classification problem. 
More advanced statistical methods that conform to tropical geometry (such as tropical support vector machines \cite{YOSHIDA202377})
could be applied instead at the cost of simplicity and interpretability. 
}

\begin{figure}[h!]
    \centering
    \includegraphics[width=.8\linewidth]{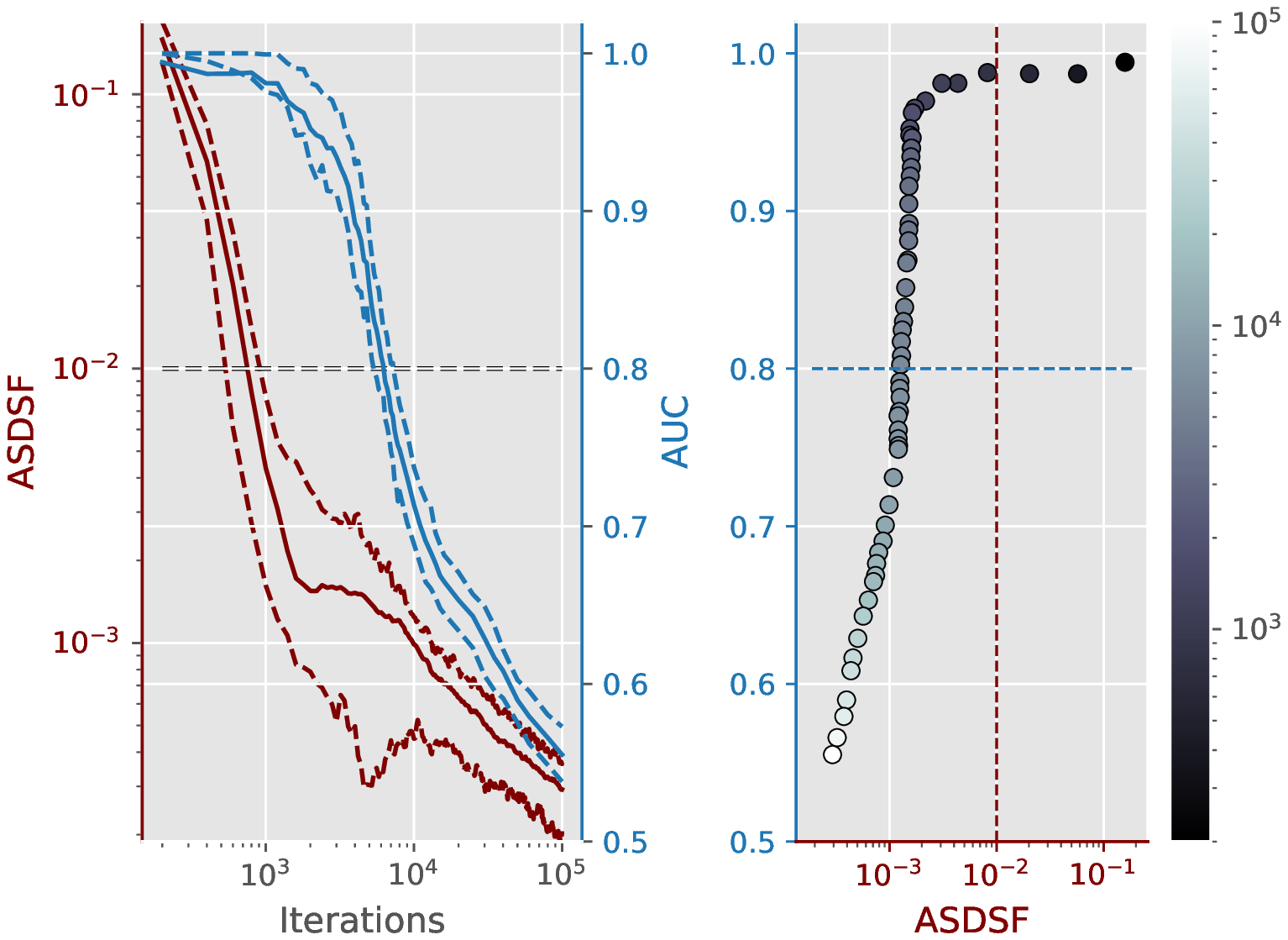}
    \caption{(Left) Average ASDSF (in red) and AUC (in blue) values plotted against the number of iterations of the MCMC chains. The coloured dashed lines correspond to the first and third quartile. 
    The grey dashed line indicates the \texttt{Mr Bayes}
    threshold for ASDSF and our provisional AUC threshold of $80\%$. (Right) ASDSF and AUC values plotted against each other, with the iterations coloured according to the colourbar and the dashed lines 
    corresponding to the thresholds for each metric.}
    \label{fig:primates}
\end{figure}

\begin{figure}[h!]
    \centering
    \includegraphics[width=.6\linewidth]{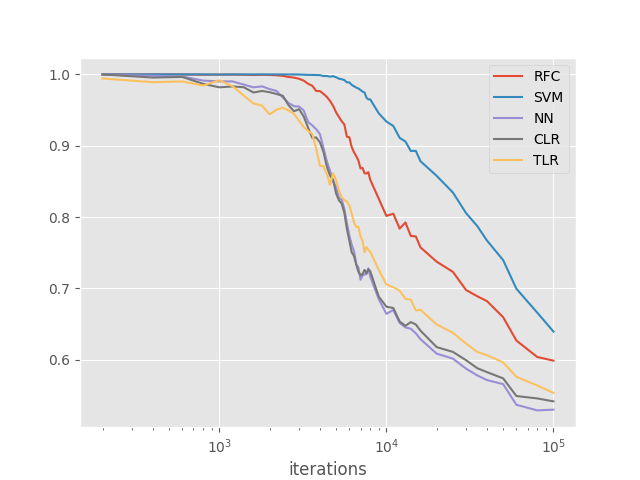}
    \caption{\color{black} Average AUC values plotted against the number of MCMC iterations for the 5 supervised learning methods considered.}
    \label{fig:mb_aucs_comparison}
\end{figure}

\section{Discussion} \label{sec:discussion}

In this paper we developed a tropical analog of the classical logistic regression model 
and considered two special cases; the
one species-tree model and two species-tree model.
In our empirical work the two-species model was most effective, but we anticipate both are potentially impactful tools for phylogenomic analysis. 
The one-species model's principal benefit is having the same 
number of parameters as the number of predictors, unlike the two-species model which has 
almost twice as many. Therefore, {\color{black}the one-species model} more readily fits the standard definition of a generalized linear model
and could generalize to a stack of GLMs to produce a {\color{black}``}tropical{\color{black}''} neural network, which is investigated in 
\cite{yoshida2023tropical}.
\par 
The two-species model implemented on data generated under the coalescent model
outperformed classical and BHV logistic regression models in terms of 
misclassification rates, AUCs and fitness of the distribution of distances to their 
centre. It was also observed that Laplacian distributions were better fitting than Gaussians, 
for both geometries.
Empirically selecting tropical distributions over Euclidean distributions suffices for the scope of this paper, but further theoretical justification of the suitability of such distributions is needed.
Moreover, further research on the generalization error for the two-species model would provide 
tighter bounds. 
\par 
Finally, the AUC metric of our model is proposed as an alternative to the ASDSF metric for MCMC convergence checking.
Our metric is more conservative and robust, taking 
branch lengths into account. 
Nonetheless, computing the ASDSF is less computationally 
intensive than running our method. 
There seems to be a tradeoff between 
the reliability of the convergence criterion tool and 
computational speed. 
Further research can shed light on the types of datasets where the ASDSF metric becomes unreliable. Then, the two metrics could complement each other, with our methods applied only when there is a good indication that ASDSF is unreliable.


\section*{Acknowledgments}

RY is partially funded by NSF Division of Mathematical Sciences: Statistics Program DMS 1916037. GA is funded by EPSRC through the STOR-i Centre for Doctoral Training under grant EP/L015692/1. 

\section*{Declarations}

Some journals require declarations to be submitted in a standardised format. Please check the Instructions for Authors of the journal to which you are submitting to see if you need to complete this section. If yes, your manuscript must contain the following sections under the heading `Declarations':

\begin{itemize}
\item Funding: RY is partially funded by NSF Division of Mathematical Sciences: Statistics Program DMS 1916037. GA is funded by EPSRC through the STOR-i Centre for Doctoral Training under grant EP/L015692/1. 
\item Conflict of interest: No conflict of interest.
\item Ethics approval: NA
\item Consent to participate: NA
\item Consent for publication: NA
\item Availability of data and materials: DRYAD with DOI: 10.5061/dryad.tht76hf65
\item Code availability: DRYAD with DOI: 10.5061/dryad.tht76hf65
\item Authors' contributions: GA contributed theoretical work and computations.  RY directed this project.  BB and JG supervised GA.
\end{itemize}

\appendix
\section{Proofs} \label{sec:proofs}
\input{intro_proofs}

\input{prerequisites}

\input{consistency_of_estimators}

\begin{proof}[{\bf Proof of Proposition 
\ref{prop:gen_error_one_species}}]
First, define $\Delta_0 = \{ C(X) \neq 1  | Y = 0 \} $. 
By definition of $C(X)$, 
\begin{align*}
    \Delta_0 &= \left\{ \left. (\sigma_0^{-1} - \sigma_1^{-1}) d_{\rm tr}(X,\Hat{\omega})
    - (e-1) \log{\left( \frac{\sigma_1}{\sigma_0}  \right)} \geq 0
    \right| Y=0 
    \right\}\\
    & = 
    \left\{ \left.
    d_{\rm tr}(X,\Hat{\omega}) \geq \alpha \sigma_0 \sigma_1 
    \right| Y=0 
    \right\}.
\end{align*}
Triangular inequality dictates that 
\begin{equation*}
    d_{\rm tr}(X,\omega^*) - d_{\rm tr}(\omega^*,\Hat{\omega}) 
    \leq 
    d_{\rm tr}(X,\Hat{\omega}) 
    \leq 
    d_{\rm tr}(X,\omega^*) + d_{\rm tr}(\omega^*,\Hat{\omega}), 
\end{equation*}
and so it follows that 
\begin{align*}
    \Delta_0 & \supseteq 
    \{ d_{\rm tr}(X,\omega^*) \geq \sigma_0 \sigma_1 \left( \alpha + \epsilon  \right) | Y = 0  \} \\ 
    \Delta_0 & \subseteq
    \{ d_{\rm tr}(X,\omega^*) \geq \sigma_0 \sigma_1 \left( \alpha - \epsilon  \right) | Y = 0  \}, 
\end{align*}
and since $Z = \sigma_0^{-1} d_{\rm tr}(X,\omega^*)|Y=0\sim F$, 
\begin{equation*}
    \mathbb{P}( Z \geq \sigma_1 (\alpha + \epsilon)) \leq
    \mathbb{P}(\Delta_0) \leq
    \mathbb{P}( Z \geq \sigma_1 (\alpha - \epsilon)),
\end{equation*}
which yields the desired result.\\
Similarly, for $\Delta_1 = \{ C(X) \neq 0 | Y = 1 \} = 
\left\{ d_{\rm tr}(X,\Hat{\omega}) \leq \sigma_0 \sigma_1 \alpha \right\}$,
\begin{align*}
    \Delta_1 & \supseteq 
    \{ d_{\rm tr}(X,\omega^*) \leq \sigma_0 \sigma_1 \left( \alpha - \epsilon  \right) | Y = 1  \} \\ 
    \Delta_1 & \subseteq
    \{ d_{\rm tr}(X,\omega^*) \leq \sigma_0 \sigma_1 \left( \alpha + \epsilon  \right) | Y = 1  \}, 
\end{align*}
and since $Z = \sigma_1^{-1} d_{\rm tr}(X,\omega^*)|Y=1\sim F$, 
\begin{equation*}
    \mathbb{P}( Z \leq \sigma_0 (\alpha - \epsilon)) \leq
    \mathbb{P}(\Delta_1) \leq
    \mathbb{P}( Z \leq \sigma_0 (\alpha + \epsilon)),
\end{equation*}
which is the desired interval. \\
For the second part of the proposition, $\Hat{\omega} = \omega^*$ and so 
$\epsilon = 0$. Hence, 
\begin{align*}
    \mathbb{P}(\Delta_0) &= 1 - F(\sigma_1 \alpha ) = 1- F( x u(x) ) \\ 
    \mathbb{P}(\Delta_1) &= F(\sigma_0 \alpha ) = F( u(x) ), \text{ where } \\ 
    x = \frac{\sigma_1}{\sigma_0} & \text{ and }
    u(x) = (e-1) \frac{ \log{x} }{x - 1}   
\end{align*}
Consider the function 
\[g(x) = 1-F(xu(x)) - F(u(x))\]
Proving that $g(x) <0 $ for all $x > 1$ is equivalent to proving 
the desired result that 
$\mathbb{P}(\Delta_0) < \mathbb{P}(\Delta_1)$ for $\sigma_1 > \sigma_0$. 
First, 
\begin{equation*}
    \lim_{x \to 1} u(x) = \lim_{x \to 1} x u(x) = e-1,
\end{equation*}
and so $\lim_{x \to 1} g(x) = 1-2F(e-1)$. 
It is a well-known fact that the median of the ${\rm Gamma}$ distribution is less than
the mean. 
Hence, for $Z \sim {\rm Gamma}(e-1,1)$ with mean $e-1$,
$F(e-1) > \frac{1}{2}$ and so 
\begin{equation}
    \label{eq:limx_g}
    \lim_{x \to 1} g(x) < 0.
\end{equation}
Finally, the derivative of $g$ is 
\begin{equation*}
    g'(x) = 
    - F'(u(x)) u'(x) - F'(xu(x)) (xu'(x) + u(x)) 
\end{equation*}
The following two inequalities 
\begin{align}
    F'(u(x))  \geq F'(xu(x)), & \label{eq:F'} \\
    u'(x) + xu'(x) + u(x) & \geq 0, \label{eq:u'}
\end{align}
imply that 
\begin{equation}
    \label{eq:g'}
    g'(x) \leq 
    - F'(xu(x)) (u'(x) + xu'(x) + u(x))
    \leq 0.
\end{equation}
From \eqref{eq:limx_g} and \eqref{eq:g'} 
it follows that $g(x) < 0$ for all $x>1$.\\
For inequality \eqref{eq:F'}, remember that 
\begin{equation*}
    F'(x) = \frac{x^{e-2} \exp{(-x)} }{\Gamma(e-1)}
\end{equation*}
and so 
\begin{align*}
    F'(u(x)) - F'(xu(x)) &=  
    F'(u(x))
    \left(
        1 - x^{e-2}  \exp(-(x-1) u(x))
    \right) \\ &=
    F'(u(x))
    \left(
        1 - x^{e-2}  \exp(-(e-1) \log(x))
    \right) \\ &= 
    F'(u(x))(1-x^{-1})>0,
\end{align*}
for all $x>1$.\\ 
For inequality \eqref{eq:u'},
\begin{equation*}
     u'(x) + xu'(x) + u(x) = \frac{e-1}{(x-1)^2} \left( x-x^{-1} -2\log{x} \right),
\end{equation*}
is a non-negative function for $x>1$
iff $v$ is a non-negative function, where
\begin{align*}
    v(x) & = x - x^{-1} - 2\log{x}, \text{ with } \\
    v'(x) & = \frac{(x-1)^2}{x^2} \geq 0
    \text{ and } 
    v(1) = 0. 
\end{align*}
Clearly, $v$ is a non-negative function
for $x>1$, so 
inequality \eqref{eq:u'} is satisfied.
\end{proof}

\begin{proof}[{\bf Proof of Proposition 
\ref{prop:gen_error}}]
    For symbolic convenience, in this proof class $0$ is referred to as class $-1$ 
    and so $Y \in \{-1,1\}$.
    Applying the triangular inequality twice, 
    \begin{align*}
        D_X &
        = d_{\rm  tr}(X , \omega^*_Y)  - d_{\rm  tr}(X , \omega^*_{-Y}) 
        \\ & \geq 
        \left( 
            d_{\rm  tr}(X , \Hat{\omega}_Y) - d_{\rm  tr}(  \omega^*_Y , \Hat{\omega}_Y )  
        \right) \\& - \left(
            d_{\rm  tr}(X , \Hat{\omega}_{-Y}) + d_{\rm  tr}(  \omega^*_{-Y} , \Hat{\omega}_{-Y} )  
        \right) \\
        &= 
        d_{\rm  tr}(X , \Hat{\omega}_Y) - d_{\rm  tr}(X , \Hat{\omega}_{-Y}) - \epsilon, 
    \end{align*}
    it follows that
    \[
    \{ C(X) \neq Y \} =
    \{d_{\rm  tr}(X , \Hat{\omega}_Y)  - d_{\rm  tr}(X , \Hat{\omega}_{-Y})
        \geq 0\}  \subseteq \{D_X \geq - \epsilon\}
    \] 
    and so 
    the generalization error has the following upper bound
    \begin{equation}
        \label{ineq:gen_error}
        \mathbb{P}( C(X) \neq Y ) 
        \leq 
        \mathbb{P} \left( D_X \geq -\epsilon \right).
    \end{equation}
    Note that if 
    $d_{\rm tr}(X,\omega_Y^*) < \Delta_\epsilon$, 
    then by the use of triangular inequality 
    \begin{align*}
         D_X & =d_{\rm  tr}(X , \omega^*_Y)  - d_{\rm  tr}
         (\omega^*_{-Y},X) 
         \\ & \leq 
         d_{\rm  tr}(X , \omega^*_Y) - \left(
            d_{\rm  tr}(\omega_{-Y}^*, \omega_Y^* ) -d_{\rm  tr}
            (\omega^*_Y,X) 
         \right)
         \\ &
         < 2\Delta_{\epsilon}-d_{\rm  tr}(\omega_1^* , \omega_{-1}^*) = -\epsilon. 
    \end{align*}
    Consequently, 
    \begin{equation}
        \label{eq:gen_error_leq_DXZX}
        \mathbb{P}( C(X) \neq Y )  \leq 
        \mathbb{P}\left( D_X \geq -\epsilon , 
        Z_X \geq \Delta_\epsilon
        \right) 
    \end{equation}
    Since the distribution of $X$ is symmetric around 
    $\omega^*_Y$, the random variable 
    $2\omega^*_Y - X$ has the same distribution and so 
    \begin{equation}
        \label{eq:gen_error_2omega-X}
        \mathbb{P}\left( D_X \geq -\epsilon , 
        Z_X \geq \Delta_\epsilon
        \right) = \mathbb{P}\left( D_{2\omega_Y^*-X} \geq -\epsilon , 
        Z_{2\omega_Y^*-X} \geq \Delta_\epsilon
        \right). 
    \end{equation}
    It will be proved that 
    \begin{align}
        Z_{2\omega_Y^*-X} &= Z_{X}, \label{eq:Z} \\ 
        D_X + D_{2\omega_Y^*-X} &\leq 0, \label{eq:D}
    \end{align}
    and so 
    $\{D_{2\omega_Y^*-X} \geq -\epsilon , 
        Z_{2\omega_Y^*-X} \geq \Delta_\epsilon\} \subseteq 
    \{D_{X} \leq \epsilon , 
    Z_{X} \geq \Delta_\epsilon\}.
    $
    Then, using equation \eqref{eq:gen_error_2omega-X}, 
    \begin{equation*}
    \mathbb{P}\left( D_X \geq -\epsilon , 
        Z_X \geq \Delta_\epsilon
        \right)
    \leq  
    \mathbb{P}\left(D_{X} \leq \epsilon , 
    Z_{X} \geq \Delta_\epsilon\right),
    \end{equation*}
    and substituting it to inequality 
    \eqref{eq:gen_error_leq_DXZX},
    \begin{align*}
        \mathbb{P}(C(X) \neq Y)) = & 
        \frac{1}{2} ( \mathbb{P}\left( D_X \geq -\epsilon , 
        Z_X \geq \Delta_\epsilon
        \right)  \\ + & \,\,\,\,
        \mathbb{P}\left( D_X \leq \epsilon , 
        Z_X \geq \Delta_\epsilon
        \right) 
	) \\ = &  
        \mathbb{P}\left( 
            Z_X \geq \Delta_\epsilon
        \right) + h(\epsilon) 
    \end{align*}
    where $h(\epsilon) = 
    \mathbb{P}(Z_X \geq \Delta_\epsilon, |D_X| \leq \epsilon)$
    is an increasing function with respect to $\epsilon$, 
    which completes the first part of the proof.\\
    Equation \eqref{eq:Z} follows from the observation that
    \[d_{\rm tr}(2\omega_Y^*-x,\omega_Y^*) = 
    d_{\rm tr}(x,\omega_Y^*).\]
    For equation \eqref{eq:D},
    \begin{align*}
        D_{2\omega_Y^*-X} + D_X &=
        Z_{2\omega_Y^*-X} 
        - d_{\rm  tr}( 2 \omega_Y^* - X , \omega_{-Y}^* )  \\&
        + Z_X
        - d_{\rm  tr}(X,\omega_{-Y}^*) \\ &
        \overset{\eqref{eq:Z}}{=}
        2Z_{2\omega_Y^*-X}  -  
            d_{\rm  tr}( 2 \omega_Y^* - X,  \omega_{-Y}^* ) -
            d_{\rm  tr}(\omega_{-Y}^*,X)
         \\ 
        & \leq 
        2Z_{2\omega_Y^*-X}  -  
            d_{\rm  tr}( 2 \omega_Y^* - X,  X ) 
        = 0,
    \end{align*}
    where the last inequality comes from the triangular inequality.
    Finally, the consistency of the learning algorithm is proved.
    Under the conditions of Theoreom \ref{prop:consistency_of_estimators}, the maximum likelihood estimator
    $\Hat{\omega} = (\Hat{\omega}_0, \Hat{\omega}_1) 
    \overset{\rm p}{\to}(\omega_0^*,\omega_1^*)$ as $n \to \infty$ 
    where $(X_1,Y_1), \dots, (X_n,Y_n)$ is the sample. 
    For the rest of the proof, the test covariate-response pair
    $(X,Y)$ is independent from the afore training sample. 
    Define the classifier, 
    \[
        C_{\omega}(x) = {\rm sgn}\left( 
            d_{\rm tr}(x,\omega_0) - d_{\rm tr}(x,\omega_1)
        \right)
    \]
    where $\omega = (\omega_0,\omega_1)$.
    The Bayes predictor is $C_{\omega_0^*,\omega_1^*}$.
    Noting that $C_{\omega_0^*,\omega_1^*}(X) = {\rm sgn}(D_X)Y$,
    the Bayes (or irreducible) error is 
    \begin{equation*}
        {\rm BE} = \mathbb{P}(  {\rm sgn}(D_X)Y \neq Y ) =
        \mathbb{P}( D_X > 0) =  
        \mathbb{P}( D_X \geq 0), 
    \end{equation*}
    since it is assumed that $\mathbb{P}(D_X =0)=0$.
    Using inequality \ref{ineq:gen_error} derived earlier, 
    it follows that the generalization error is
    bounded by
    \begin{equation*}
        \mathbb{P}( D_X \geq 0) = 
        {\rm BE} 
        \leq \mathbb{P}(C_{\Hat{\omega}}(X) \neq Y)
        \leq \mathbb{P}(D_X \geq -\epsilon(\Hat{\omega})),
    \end{equation*}
    where $\epsilon(\Hat{\omega}_0,\Hat{\omega}_1)=
    d_{\rm tr}(\omega_0,\omega_0^*) + 
    d_{\rm tr}(\omega_1,\omega_1^*) \overset{\rm p}{\to} 0$.
    as the training sample size $n \to \infty$ according to 
    Theorem \ref{prop:consistency_of_estimators}.
    The complementary CDF of $D_X$, defined as 
    $$F_{D_X}^C(x) = \mathbb{P}(D_X \geq x),$$ 
    is a continuous function and so it follows that 
    $F_{D_X}^C(\epsilon(\Hat{\omega})) \overset{\rm p}{\to} 
    F_{D_X}^C(0) = {\rm BE}$ as $n \to \infty$.
    From the probability squeeze theorem, 
    \begin{equation*}
        \mathbb{P}(C_{\Hat{\omega}}(X)\neq Y|
        (X_1,Y_1), \dots, (X_n,Y_n) ) 
        \overset{\rm p}{\to} 
        {\rm BE} 
        \text{ as }
        n \to \infty.
    \end{equation*}
    This concludes the proof of the consistency of the algorithm.
\end{proof}

\begin{proof}[{\bf Proof of Proposition 
\ref{prop:consistency_of_Fermat_Weber_points}}]
Consider the random variable 
$d_{\rm tr}(X,\alpha)$. From the triangular inequality 
\begin{equation*}
    d_{\rm tr}(X,\alpha) \leq 
    d_{\rm tr}(X,\omega^*) + 
    d_{\rm tr}(\alpha,\omega^*),
\end{equation*}
it is deduced that $d_{\rm tr}(X,\alpha)$ is integrable, bounded
above by an integrable random variable. 
\par 
Now consider the function 
$F: \trtor \to \R$,
\[
    F(x) = d_{\rm tr}(x,\omega) + 
    d_{\rm tr}(2\omega^* - x,\omega) - 
    2d_{\rm tr}(x,\omega^*).
\]
Noting that 
$d_{\rm tr}(2\omega^* - x,\omega) 
= d_{\rm tr}(x,2\omega^* - \omega)$,
it follows that $F(X)$ is integrable as the sum of integrable 
random variables.
\par 
From triangular inequality 
and the fact that $d_{\rm tr}(2\omega^* - x,x) = 2d_{\rm tr}(x,\omega^*)$
it follows that 
$F(x) \geq 0$ for all $x \in \trtor$.
Furthermore, $F(\omega^*)>0$ and since $F$ is continuous, there 
exists a neighbourhood $U$ that contains $\omega^*$ such that 
$F(x) > 0$ for all $x \in U$. Moreover, 
the function has positive density in a neighbourhood $V$ 
that contains the centre $\omega^*$. Therefore, there exists 
a neighbourhood $W = U \cap V$ such that $F(x) > 0$ for all 
$x \in W$ and $\mathbb{P}(X \in W) > 0$. Hence, since 
$F(X) \geq 0$, 
\begin{equation*}
    \mathbb{E}(F(X)) 
    \geq 
    \mathbb{E}(F(X)| X \in W) \mathbb{P}(X \in W ) > 0. 
\end{equation*}
In other words, 
\begin{equation}
    \label{ineq:EF>0}
    \mathbb{E}\left(d_{\rm tr}(X,\omega)\right) + 
    \mathbb{E}\left(d_{\rm tr}(2\omega^*-X,\omega)\right) >  
    2 \mathbb{E}(d_{\rm tr}(x,\omega^*))
\end{equation}
Moreover, consider the isometry $y = 2\omega^*-x$ 
and note that for symmetric probability 
density functions around $\omega^*$, 
$f(\omega^*-\delta) = f(\omega^* + \delta)$ and so for 
$\delta = \omega^* - x$, we have $f(y)=f(x)$. 
Applying this transformation to the following integral yields
\begin{align}
    \mathbb{E}(d_{\rm tr}(2\omega^* - X,\omega)) &= 
    \int_{\trtor} d_{\rm tr}(2\omega^* - x,\omega) f(x) \, dx  
    \label{eq:2omega-x} \\ & =
    \int_{\trtor} d_{\rm tr}(y,\omega) f(y) \, dy 
    = \mathbb{E}(d_{\rm tr}(X,\omega)). \nonumber
\end{align}
Combining equation \eqref{eq:2omega-x} with inequality 
\eqref{ineq:EF>0} shows that the function 
$Q(\omega) = \mathbb{E}(d_{\rm tr}(X,\omega))$ 
has a global minimum at $\omega^*$.
\par 
From Theorem \ref{thm:uniform_law_of_large_numbers} (uniform
law of large numbers), set 
$f(x,\omega) = d_{\rm tr}(x,\omega)$ and observe that 
$f(x,\omega)$ is always continuous w.r.t. $\omega$. 
Setting 
$r(x) = \sup_{\omega \in \Omega} d_{\rm tr}(x,\omega)$, which 
is finite since $\Omega$ is compact, observe that
\[
    r(x) := \sup_{\omega \in \Omega} d_{\rm tr}(x,\omega)
    \leq d_{\rm tr}(x,\omega^*) + 
    \sup_{\omega \in \Omega} d_{\rm tr}(\omega,\omega^*).
\]
Since $\Omega$ is compact, the second term is finite and hence 
$r(X)$ is integrable, since $d_{\rm tr}(X,\omega^*)$ is
integrable. All conditions of the theorem are satisfied so 
$Q(\omega) = \mathbb{E}(d_{\rm tr}(x,\omega))$ is continuous 
with respect to $\omega$ and 
\begin{equation*}
        \sup_{\omega \in \Omega} 
        |Q_n(\omega) - Q(\omega)| 
        \overset{p}{\to}
        0 \text{ as }
        n \to \infty, 
\end{equation*}
where $Q_n(\omega) = n^{-1} \sum_{i=1}^n d_{\rm tr}(X_i,\omega)$. 
Since $Q(\omega)$ has a unique minimum at $\omega^*$, all 
conditions of Theorem 
\ref{thm:convergence_of_theta_n} are satisfied and so 
$\Tilde{\omega}_n \to \omega^*$ as $n \to \infty$.
\end{proof}

\begin{proof}[{\bf Proof of Proposition \ref{prop:vanishing_gradient}}]
\begin{enumerate}[i.]
    \item If $\omega - X_i$ has a unique maximum 
    $M_i = \argmax_j\{\omega_j - (X_i)_j\}$ and 
    unique minimum 
    $m_i = \argmin_j\{\omega_j - (X_i)_j\}$, then the 
    gradient is
    \begin{equation}
        \label{eq:gradient_f}
        (\nabla f(x))_j = |\{i: M_i=j\}| - |\{i: m_i=j\}|.
    \end{equation}
    For the converse, assume that the gradient is 
    well-defined. From equations 
    \eqref{eq:dtr_pos}--\eqref{eq:dtr_neg} and following the 
    first few sentences of Lemma 
    \ref{lemma:uniqueness_of_optimum}
    \[d_{\rm tr}(x+\epsilon E_j,y) +
    d_{\rm tr}(x-\epsilon E_j,y) - 
    2d_{\rm tr}(x,y) = \epsilon s_j(x-y), 
    \]
    where $s_j$ is defined in equation \eqref{eq:s_i} of 
    Lemma \ref{lemma:uniqueness_of_optimum}.
    Consequently,
    \[
    f(x+\epsilon E_j) + f(x- \epsilon E_j) -
    2f(x) = \epsilon \sum_{i=1}^n s_j(X_i - \omega_i)
    \]
    Since $f$ has a well-defined gradient, 
    $\sum_{i=1}^n s_j(X_i - \omega) = 0$
    i.e. $s_j(X_i - \omega) = 0$ for all 
    $(i,j) \in [n] \times [e]$. This can only happen iff 
    $X_i - \omega$ has unique maximum and minimum component 
    for all $i \in [n]$.
    \item Using equation \eqref{eq:gradient_f}, the gradient 
    of $f$ vanishes at $x=\omega$ if and only if
    \begin{equation}
        \label{eq:vanishing_gradient}
        |\{i: M_i=j\}| = |\{i: m_i=j\}|.
    \end{equation}
    Moreover,
    \begin{align*}
        f(\omega+v) &= \sum_{i=1}^n 
        \max_k \left\{ \omega_k - (X_i)_k + v_k \right\} - 
        \min_k \left\{ \omega_k - (X_i)_k + v_k \right\}\\ 
        & \geq 
        \sum_{i=1}^n 
        \omega_{M_i} - (X_i)_{M_i} + v_{M_i} - 
        \omega_{m_i} + (X_i)_{m_i} - v_{m_i} \\ 
        & = 
        f(\omega) + \sum_{i=1}^n v_{M_i} - v_{m_i}
    \end{align*}
    Finally, note that because of equation 
    \eqref{eq:vanishing_gradient}, 
    \begin{align*}
        \sum_{i=1}^n v_{M_i} &= \sum_{j=1}^e 
        v_j |\{i \in [n]:M_i=j\}| \\ &
        \overset{\eqref{eq:vanishing_gradient}}{=} 
        \sum_{j=1}^e 
        v_j |\{i \in [n]:m_i=j\}| = 
        \sum_{i=1}^n v_{m_i},
    \end{align*}
    and so $f(\omega+v) \geq f(\omega)$ for all 
    $v \in \trtor$. 
\end{enumerate}

\end{proof}



\section{Space of ultrametrics} \label{sec:ult_thm}
\begin{theorem}[explained in \cite{AK,10.1093/bioinformatics/btaa564}] \label{thm}
Suppose we have a {\em classical} linear subspace $L_m \subset \RR^e$ defined by the linear
equations $x_{ij} - x_{ik} + x_{jk}=0$ for $1\leq i < j <k \leq m$. Let $\mbox{Trop}(L_m)\subseteq \RR^e/\RR {\bf 1}$ be the {\em tropicalization} of the linear space $L_m \subset \RR^e$, that is, 
classical operators are replaced by tropical ones (defined in Section \ref{Sec:tro:inner:prod} in the supplement)
in the equations defining the linear subspace $L_m$, so that all points $(v_{12},v_{13},\ldots, v_{m-1,m})$ in $\Trop(L_m)$ satisfy the condition that 
\[
\max_{i,j,k\in [m]}\{v_{ij},v_{ik},v_{jk}\}.
\]
is attained at least twice.
Then, the image of $\Tn$ inside of the tropical projective torus $\RR^e/\RR {\bf 1}$ is equal to $\Trop(L_m)$. 
\end{theorem}

\section{Tropical Arithmetics and Tropical Inner Product}\label{Sec:tro:inner:prod}

In tropical geometry, addition and multiplication are different than
regular arithmetic. 
The arithmetic operations are performed
in the
max-plus tropical semiring $(\,\mathbb{R} \cup \{-\infty\},\oplus,\odot)\,$
as defined in \cite{pin1998tropical}. 
\begin{definition}[Tropical Arithmetic Operations]
  In the tropical semiring, the basic tropical
  arithmetic operations of addition and multiplication are defined as:
  $$a \oplus b := \max\{a, b\}, ~~~~ a \odot b := a + b, ~~~~\mbox{  where } a, b \in \mathbb{R}\cup\{-\infty\}.$$
  The element $-\infty$ ought to be included as it is the 
  identity element of tropical addition.
  Tropical subtraction is not well-defined and tropical division is classical subtraction.
\end{definition}

The following definitions are necessary for the definition of the tropical inner product
\begin{definition}[Tropical Scalar Multiplication and Vector Addition]
  For any scalars $a,b \in \mathbb{R}\cup \{-\infty\}$ and for any vectors 
  $v,w \in (\mathbb{R}\cup\{-\infty\})^e$, 
  where $e \in \mathbb{N}$, 
    \begin{align*}
        a \odot & v:= (a + v_1,  \ldots ,a + v_e), \\ 
        a \odot v  \oplus b \odot w &:= (\max\{a+v_1,b+w_1\}, \ldots, \max\{a+v_e,b+w_e\}).
    \end{align*}
\end{definition}
From the definitions above, it follows that the tropical inner product is
$\omega^T \odot x = \max \{ \omega + x \}$ 
for all vectors $\omega,x \in \trtor$.
In classical logistic regression a linear function in the form of a classical inner product 
$h_{\omega}(x) = \omega^T x, \, \omega \in \R^n $ is used.
The tropical symbolic equivalent is 
\begin{equation}
    \label{eq:h_omega_trop}
    h_{\omega}(x) = \omega^T \odot x = \max_{l \in [e]} \{ \omega_l + x_l \}.
\end{equation}
This expression is not well-defined, since the statistical parameter and covariate
vectors $\omega, u \in \mathbb{R}^e / \mathbb R {\rm 1} $ are only defined up to 
addition of a scalar multiple of the vector $(1,\dots,1)$.
To resolve this issue, we fix 
\begin{equation}
    \label{eq:fix_min}
    - \min_{l \in [e]} \{ \omega_l + x_l  \} = c, 
\end{equation}  
where $c \in \R$ is a constant for all observations. 
Combining equations \eqref{eq:fix_min}, \eqref{eq:h_omega_trop}, and the definition 
of tropical distance \eqref{eq:tropmetric}, 
\begin{equation*}
    h_{\omega}(x) = d_{\rm tr}(x,-\omega) - c. 
\end{equation*}
For simplicity, under the transformation $-\omega \to \omega$ the expression becomes
\begin{equation}
    \label{eq:h_omega}
    h_{\omega}(x) = d_{\rm tr}(x,\omega) - c. 
\end{equation}

\section{Tropical Logistic Regression Algorithm} \label{sec:alg}

\begin{algorithm}
\caption{One-species tropical logistic regression}\label{alg:one}
\begin{algorithmic}
\State {\bf Input: } distance matrix $D \in \mathbb{R}_+^{N \times e}$, labels $Y \in \{0,1\}^{N}$ 
\State $\Tilde{\omega} = {\rm FW\_point}(D)$
\State $\Hat{\sigma}_0, \Hat{\sigma}_1 = \argmax_{\sigma_0,\sigma_1 > 0} l(\Tilde{\omega},\sigma_0,\sigma_1)| D,Y)$ with root solving.
\State {\bf Output: } $(\Tilde{\omega},\Hat{\sigma}_0,\Hat{\sigma}_0 )$
\end{algorithmic}
\end{algorithm}

\begin{algorithm}
\caption{Two-species tropical logistic regression}\label{alg:two}
\begin{algorithmic}
\State {\bf Input: } distance matrix $D \in \mathbb{R}_+^{N \times e}$, labels $Y \in \{0,1\}^{N}$ 
\State $\Tilde{\omega}_0 = {\rm FW\_point}(D[Y==0])$
\State $\Tilde{\omega}_1 = {\rm FW\_point}(D[Y==1])$
\State $\Hat{\sigma} = \argmax_{\sigma > 0} l(\Tilde{\omega}_0,\Tilde{\omega}_1,\sigma| D,Y)$ with root solving.
\State {\bf Output: } $(\Tilde{\omega}_0,\Tilde{\omega}_1,\Hat{\sigma})$
\end{algorithmic}
\end{algorithm}

\section{Fermat-Weber Point Visualization} \label{sec:fw_point_illustration}
As noted in Section \ref{sec:optimization}, the
gradient method is much faster than 
linear programming. Unfortunately, there is no guarantee
that it will guide us to a Fermat-Weber point. 
However, in practice, the gradient method tends to work 
well. Figure \ref{fig:fw_vis} illustrates just that.
Given, ten datapoint $X_1, \dots, X_{10} \in 
\mathbb{R}^3/\mathbb{R}{\bf 1} \cong \mathbb{R}^2$,
the Fermat-Weber set is found to be a trapezoid.  
This is in agreement with 
\cite{lin2018tropical}, which 
states that all Fermat-Weber sets are 
classical polytopes. The two-dimensional gradient 
vector, plotted as a vector field in Figure 
\ref{fig:fw_vis}, 
always points towards the Fermat-Weber set. 
Therefore, the gradient algorithm should always guide us
to a Fermat-Weber point. 
\begin{figure}[h]
    \centering
    \includegraphics[width=.9\linewidth]{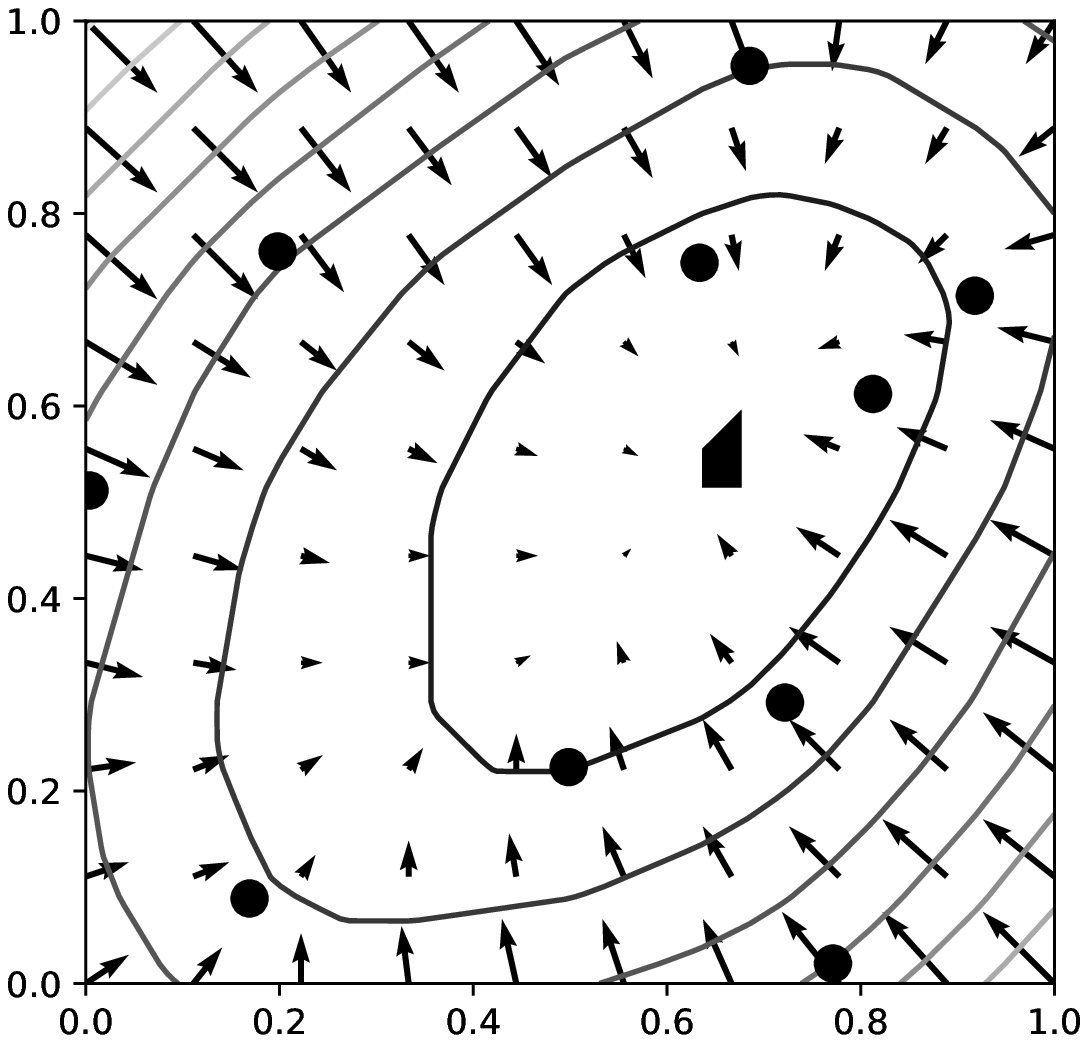}
    \caption{Visualization of the function 
    $f(\omega)= \sum_{i=1}^{10} d_{\rm tr}(X_i,\omega)$
    for $X_i$. The black circles are
    the datapoints $X_1, \dots, X_{10}$, 
    the solid lines are contours of $f$, the vector
    field is the gradient and the small black trapezoid
    at $(0.65,0.55)$ is the Fermat-Weber set.
    }
    \label{fig:fw_vis}
\end{figure}
\section{MLE Estimator for $\sigma$}
If $Z_i \overset{\rm iid}{\sim} {\rm Gamma}(n, k)$,  
where $n$ is constant and $k$ is a statistical parameter, then 
it is well-known that
the maximum likelihood estimator is 
\begin{equation}
    \label{eq:bar_k}
    \Hat{k} = \Bar{Z}/n,
\end{equation}
where $\Bar{Z}$ is the sample average. 
In our case $Z_i = d(X_i,\omega^*)$ and $k= i \sigma^i$.
From Proposition \ref{prop:tropical_radius_dist}, 
$Z_i \sim {\rm Gamma}(n/i,i\sigma^i)$ 
and by substituting these parameters in equation \ref{eq:bar_k},
it follows that the MLE for $\sigma$ is
\[
    \Hat{\sigma}^i = \Bar{Z} / n,
\]
where $\Bar{Z}$ is the average distance of the covariates (gene trees) 
from their mean (species tree). 
This results holds for all $i \in \mathbb{N}$ and both Euclidean and 
tropical metrics. The only difference is that 
for Euclidean spaces $X \in \mathbb{R}^e$ and so $n=e$, while for 
the tropical projective torus $\trtor$, $n=e-1$.
\section{Approximate BHV Logistic Regression}
Similar to the tropical Laplace distribution, in \cite{BHV} the following 
distribution was considered
\begin{equation*}
    f_{\lambda,\omega}(x) = K_{\lambda,\omega} \exp\left(-\lambda d_{\rm BHV}(x,\omega) \right),
\end{equation*}
where $\lambda = 1/\sigma$ is a concentration/precision parameter,
$d_{\rm BHV}$ is the BHV metric and $K_{\lambda,\omega}$ is the normalization 
constant that depends on $\lambda$ and $\omega$.
We consider an adaptation of the two-species model for this metric, where 
the data from the two classes have the same concentration rate but different 
centre.
If $X|Y \sim f_{\lambda,\omega^*_Y}$, then 
\begin{equation}
    \label{eq:bhv_logistic_regression}
    h_{\omega_0,\omega_1}(x) =
    \lambda \left(
        d_{\rm BHV}(x,\omega_0^*) -  d_{\rm BHV}(x,\omega_1^*)
    \right) + 
    \log{\frac{K_{\lambda,\omega^*_0 }}{K_{\lambda,\omega^*_1}}}.
\end{equation}
Unlike in the tropical projective torus or the euclidean space, in the BHV space
$K_{\lambda,\omega^*_0 } \neq K_{\lambda,\omega^*_1}$, because the space
is not translation-invariant.
However, if we assume that the two centres are far away from trees with 
bordering topologies, it may be assumed that the trees are mostly distributed 
in the Euclidean space and as a result 
$K_{\lambda,\omega^*_0 } \approx K_{\lambda,\omega^*_1}$. 
Under this assumption, equation \eqref{eq:bhv_logistic_regression}
becomes
\begin{equation*}
    h_{\omega_0,\omega_1}(x) \approx
    \lambda \left(
        d_{\rm BHV}(x,\omega_0^*) -  d_{\rm BHV}(x,\omega_1^*)
    \right). 
\end{equation*}
Therefore, the classification/decision boundary for the BHV is the 
BHV bisector $d_{\rm BHV}(x,\omega_0^*) =  d_{\rm BHV}(x,\omega_1^*)$
and the most sensible classifier is 
\begin{equation*}
    C(x) = \mathbb{I}\left(d_{\rm BHV}(x,\omega_0^*) >  
    d_{\rm BHV}(x,\omega_1^*)\right),
\end{equation*}
where $ \mathbb{I}$ is the indicator function.

\section{Graphs for Simulated Data under the Multi-Species Coalescent Model for different $R$}
\label{sec:coalescent_graphs}
\begin{figure}[h]
    \centering
    \includegraphics[width=.49\linewidth]
    {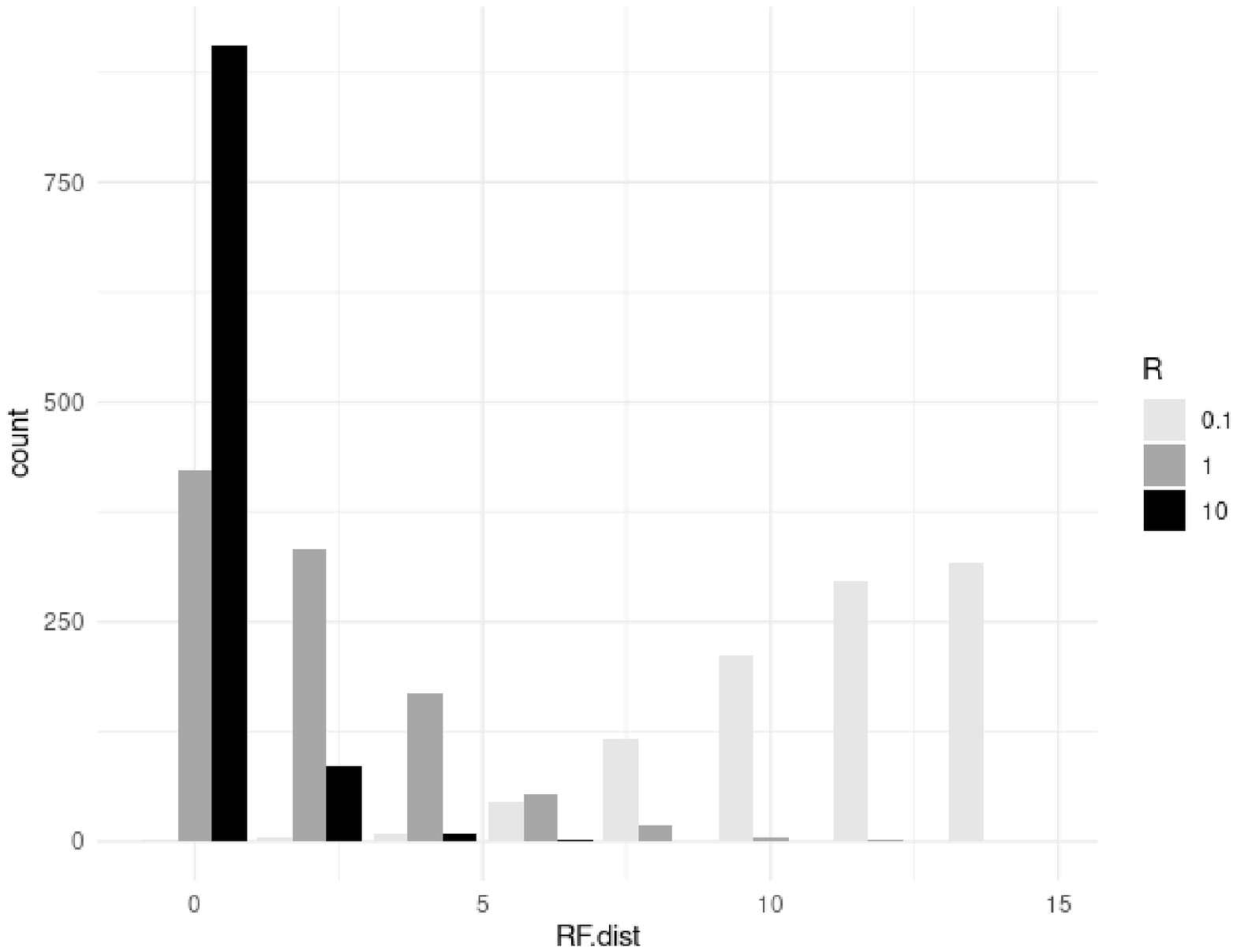}
    \includegraphics[width=.49\linewidth]
    {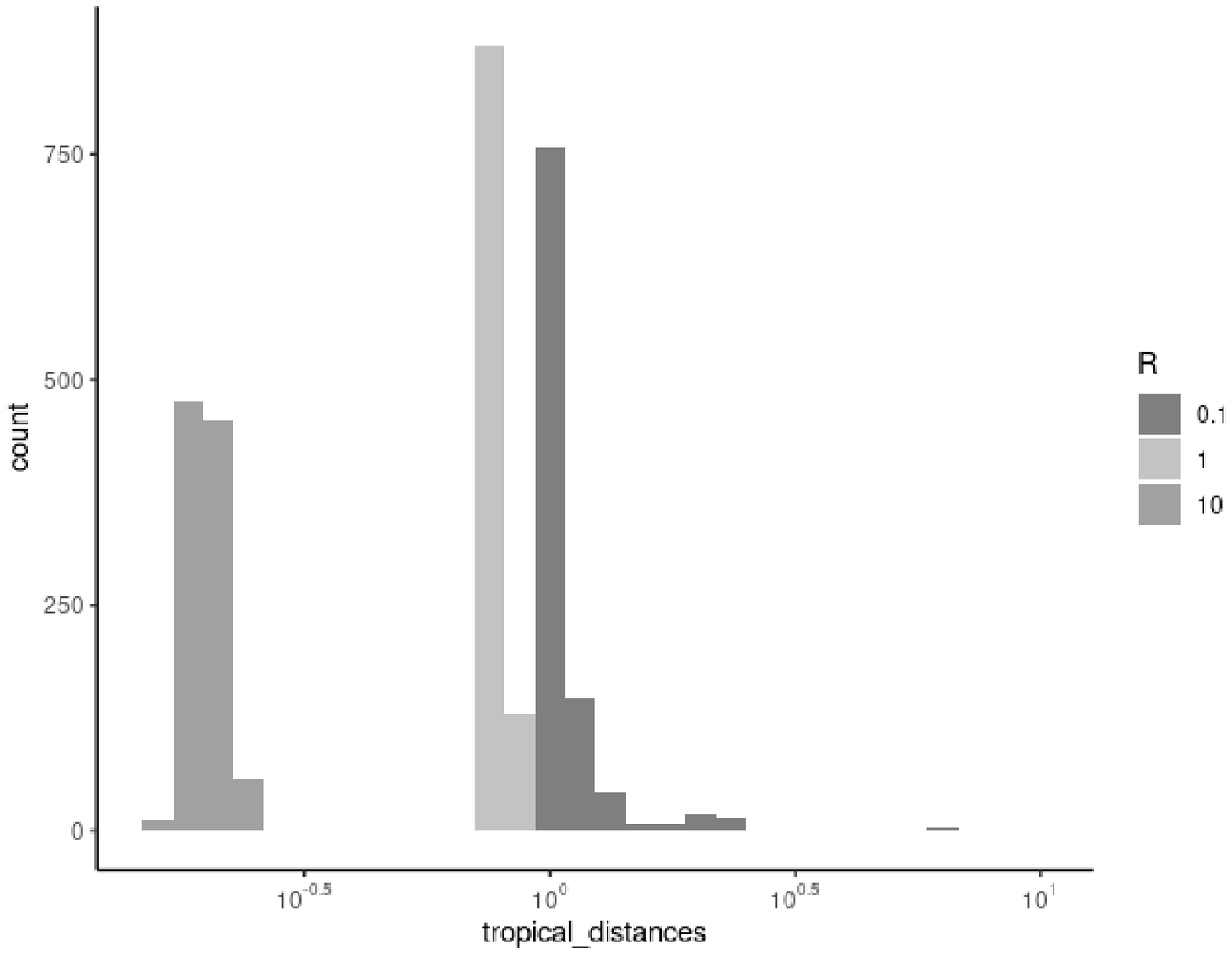}
    \caption{(left) Robinson-Foulds distances 
    and (right) tropical distances of inferred species
    trees $\Hat{\omega}$ from the 
    actual species trees $\omega^*$ for 
    $R=0.1,1,10$.}
    \label{fig:dist_hists}
\end{figure}

\begin{figure}
    \centering
    \includegraphics[width=.7\linewidth]{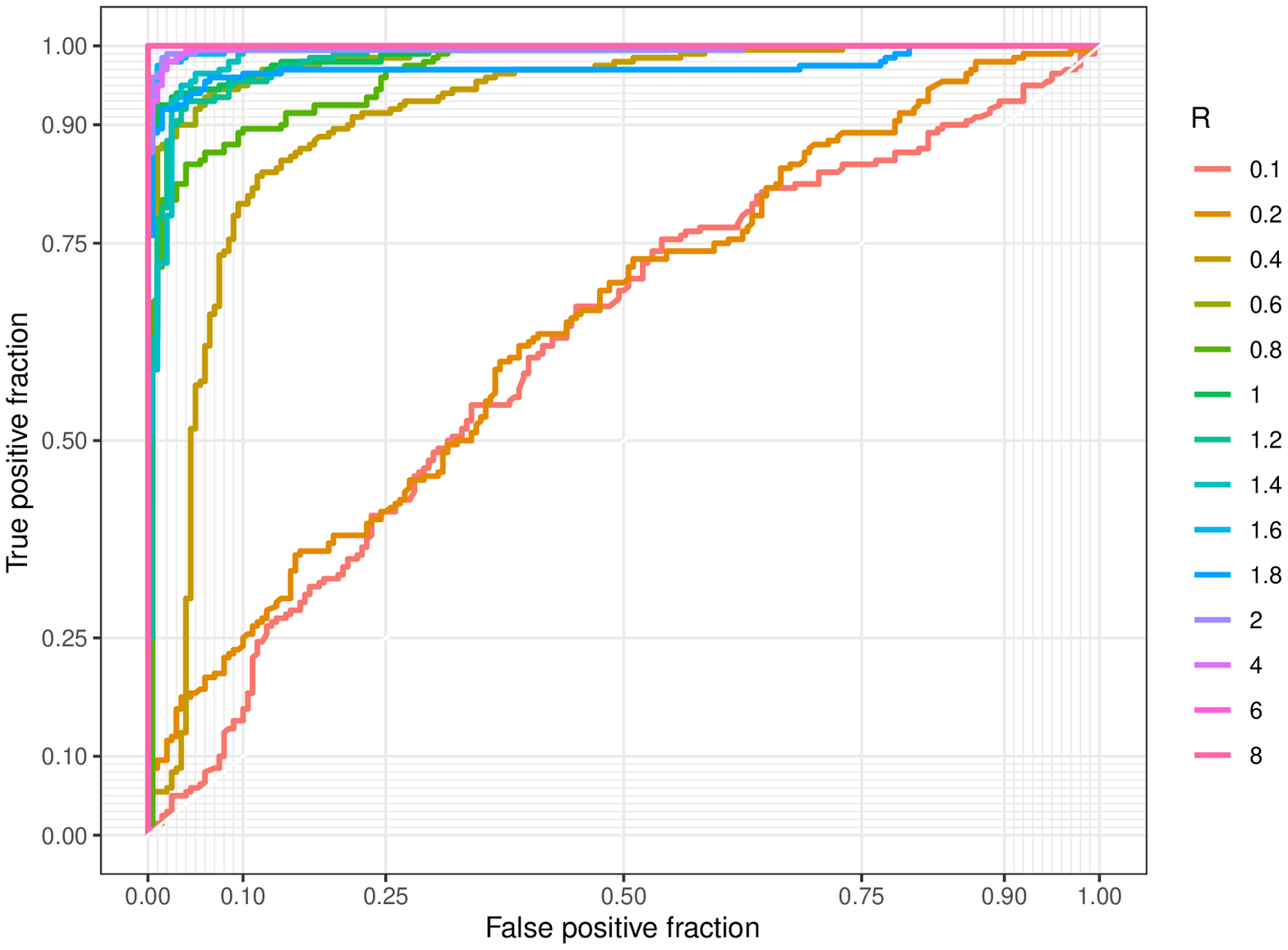}
    \caption{ROC curves for the tropical logistic regression with different values of $R$.  Higher the value of $R$ is the closer an estimated ROC curve for the tropical logistic regression model gets to the point $(0, 1)$. }
    \label{fig:rocs}
\end{figure}

\bibliographystyle{plain}
\bibliography{references}


\end{document}

%% file: intro_proofs.tex
\begin{proof}[{\bf Proof of Lemma \ref{lemma:p_bern}}]
A simple application of the Bayes rule for continuous
random variables yields 
\begin{align*}
    p(x) = \mathbb{P}(Y=1|X=x) &= 
    \frac{f_1(x) \mathbb{P}(Y=1)}
    {f_0(x) \mathbb{P}(Y=0) + f_1(x) \mathbb{P}(Y=1)} \\
    &= 
    \frac{1}
    {1 + \frac{f_1(x) (1-r)}{f_0(x) r} } = S(h(x)). 
\end{align*}
\end{proof}

\begin{proof}[{\bf Proof of Proposition \ref{prop:p_functional}}]
    The expected log-likelihood is expressed as 
    \begin{align*}
    \mathbb{E}(l) & = 
    \mathbb{E}\left(
        Y \log(p(X)) + (1-Y) \log(1-p(X))
    \right)
    \\ & = 
    \mathbb{P}(Y=1) \int_{\mathbb{R}^n}
    f_1(x)
    \log(p(x)) \, dx \\&
    + 
    \mathbb{P}(Y=0) \int_{\mathbb{R}^n}
    f_0(x) 
    \log(1-p(x)) \,
    dx \\ & = 
    \int_{\mathbb{R}^n} L(x,p(x)) \, dx ,
\end{align*}
where 
$L(x,p) = r f_1(x) \log(p) + (1-r) f_0(x) \log(1-p)$
is treated as the Lagrangian. 
The Euler-Lagrange equation can be generalized to 
a several variables 
(in our case there are $n$ variables). 
Since there are no derivatives of $p$, the stationary 
functional satisfies $\partial_p L = 0$, which yields 
the desired result. 
\end{proof}

\begin{proof}[{\bf Proof of Proposition \ref{prop:tropical_radius_dist}}]
The pdf of $X$ is 
\begin{equation*}
    f_{\omega}(x) = \frac{1}{C_{\alpha}} 
    \exp\left(-\alpha^i \frac{d^i(x)}{i} \right), x \in \mathbb{R}^n 
\end{equation*}
where $\alpha = \sigma^{-1}$ is the precision. 
Using the variable transformation $y = \alpha x$ with Jacobian 
$1/\alpha^n$
and 
remembering that $\alpha d(x) = d(y)$, 
\begin{equation*}
    C_{\alpha} = \int_{\mathbb{R}^n} \exp\left(
        -\alpha^i \frac{d^i(x)}{i}
    \right) \, dx =
    \int_{\mathbb{R}^n} \exp\left(
        - \frac{d^i(x)}{i}
    \right) \, \frac{dy}{\alpha^n}= 
    \frac{C_1}{\alpha^n}.
\end{equation*}
The moment generating function of $d^i(X)$ is 
\begin{align*}
    M_{d^i(X)} &= \int_{\mathbb{R}^n} \exp \left(
        z d^i(x)
    \right) \frac{\exp \left(
        - \alpha^i \frac{d^i(x)}{i} 
    \right)}{C_{\alpha}} \, dx \\ &= 
    \frac{C_{\sqrt[i]{\alpha^i/i - z} }}{C_\alpha} = 
    \frac{1}{\left(\sqrt[i]{1 - i\sigma^i z} \right)^n},
\end{align*}
which coincides with the MGF of $\Gamma(n/i,i\sigma^i)$.
\end{proof}

\begin{proof}[{\bf Proof of Proposition \ref{prop:normalizing_factor}}]
From the proof of Proposition 
\ref{prop:tropical_radius_dist}, 
it was established that 
the normalizing constant is
$C_{\sigma_Y} = C_{1} \sigma_Y^{e-1}$ for the tropical projective 
torus, whose dimension is $n=e-1$. 
\par 
The volume of a unit tropical sphere in 
the tropical projective torus 
$\mathbb{R}^e/\mathbb{R}{\bf 1}$ is equal to $e$.
If the tropical radius is $r$, then the volume is 
$e r^{e-1}$ and hence the surface area is 
$e (e-1) r^{e-2}$. Therefore, 
\begin{align*}
    C_1 &= \int_{\mathbb{R}^e/\mathbb{R}{\bf 1}} 
    \exp{(-d_{\rm tr}(x,{\bf 0}))} dx \\ &= 
    \int_0^{\infty} e (e - 1) r^{e-2} 
    \exp{(-r)}
    \, dr \\ & = 
    e(e-1) \Gamma(e-1) = e!
\end{align*}
It follows that the normalizing constant is 
$C_{\sigma_Y} = e! \sigma_Y^{e-1}$. 
\end{proof} 

\begin{proof}[{\bf Proof of Corollary \ref{cor}}]
    Suppose that $X$ comes from the Laplace or the 
    Normal distribution, whose pdf is proportional to 
    $\exp{(-d^i(x,\omega^*)/(i\sigma^{i})) }$ for $i=1$ and $2$ respectively,
    for all $x \in \R^{n}$ 
    where $d$ is the Euclidean metric.
    Then, $X-\omega^*$ has a distribution proportional to 
    $\exp{(-d^i(x,{\bf 0})/(i\sigma^{i}))}$. 
    Clearly, $\alpha d(x,{\bf 0}) = d(\alpha x,{\bf 0})$ and so from Proposition
    \ref{prop:tropical_radius_dist}, it follows that
    $d^i(X-\omega^*,{\bf 0}) = d^i(X,\omega^*) \sim i\sigma^i {\rm Gamma}(n/i)$.
    Note that for the normal distribution ($i=2$), 
    $d^i(X,\omega^*) \sim \sigma^2 \chi_{n/2}$.
    The same argument applies for tropical Laplace and tropical Normal 
    distributions, where the metric is tropical ($d=d_{\rm tr}$), 
    the distribution is defined on $\trtor \cong \mathbb{R}^{e-1}$ 
    and the dimension is hence $n=e-1$.
\end{proof}

%% file: prerequisites.tex
{\bf Prerequisites for proof of Theorem \ref{prop:consistency_of_estimators}}
\begin{theorem}  (Theorem 4.2.1 in \cite{bierens1996topics})
    \label{thm:convergence_of_theta_n}
    Let $(Q_n(\theta))$ be a sequence of random functions on a compact set 
    $\Theta \subset \mathbb{R}^m$ such that for a continuous real function
    $Q(\theta)$ on $\Theta$, 
    \begin{equation*}
        \sup_{\theta \in \Theta} 
        |Q_n(\theta) - Q(\theta)| 
        \overset{p}{\to}
        0 \text{ as }
        n \to \infty.
    \end{equation*}
    Let $\theta_n$ be any random 
    vector in $\Theta$ satisfying 
    $Q_n(\theta_n) = 
    \inf_{\theta in \Theta}
    Q_n(\theta)$
    and let $\theta_0$ be a unique 
    point in $\Theta$ such that 
    $Q(\theta_0) = 
    \inf_{\theta \in \Theta}
    Q(\theta)$. 
    Then $\theta_n 
    \overset{p}{\to} 
    \theta_0$.
\end{theorem}

\begin{theorem}  (Lemma 2.4 in \cite{newey1994large})
    \label{thm:uniform_law_of_large_numbers}
    If the data $z_1,\dots,z_n$
    are independent and identically distributed, 
    the parameter space $\Theta$ is compact, 
    $f(z_i,\theta)$ is continuous at each $\theta \in \Theta$
    almost surely and there is $r(z) \geq |f(z,\theta)|$ for all 
    $\theta \in \Theta$ and $\mathbb{E}(r(z))<\infty$, then 
    $\mathbb{E}(f(z,\theta))$ is continuous and 
    \begin{equation*}
        \sup_{\theta \in \Theta}
        \left|
            n^{-1} \sum_{i=1}^{n} f(z_i,\theta) - 
            \mathbb{E}(f(z,\theta))
        \right|
        \overset{p}{\to}
        0.
    \end{equation*}
\end{theorem}

%% file: consistency_of_estimators.tex
\begin{lemma} 
\label{lemma:derivative_of_tropical_distance}
Consider two points $x,y \in
\mathbb{R}^e/\mathbb{R}{\bf 1} $. 
There exists $\eta > 0$ 
such that 
\begin{equation*}
    d_{\rm tr}(x+\epsilon E_i,y) = 
    d_{\rm tr}(x,y) + 
    \epsilon \phi_i(x-y)
    ,\, 
    \forall \epsilon \in [0,\eta], \,
    \forall i \in [e], 
    \text{ where }
\end{equation*}
\begin{equation}
    \label{eq:phi}
    \phi_i(v) = \begin{cases}
        1, & \text{if } 
        v_i \geq v_j \, \forall j \in [e] \\
        -1, & v_i < v_j \, \forall j \in [e] \backslash \{ i \} \\
        0, & \text{otherwise } 
        \end{cases},
\end{equation}
and $E_i \in \mathbb{R}^e/\mathbb{R}{\bf 1}$ is a vector with $1$ in the $i$-th coordinate and $0$ elsewhere. 
\end{lemma}
\begin{proof}
By setting $v:=x-y$, $M := \max_{j \in [e]} \{ v_j \}$ and 
$m := \min_{j \in [e]} \{ v_j \}$, 
\begin{align*}
     d_{\rm tr}(x,y) &= M - m \\
     d_{\rm tr}(x+\epsilon E_i,y) &= 
    \max_{j\in[e]} \{ v_j + \epsilon \delta_{ij} \} - 
    \min_{j\in[e]} \{ v_j + \epsilon \delta_{ij} \},
\end{align*}
where $\epsilon \geq 0$, and $\delta_{ij}=\mathbb{I}(i=j)$ with $\mathbb{I}$
being the indicator function.
Three separate cases 
are considered.
\begin{enumerate}[i.]
    \item If $v_i = M$, then
    \begin{align}
        \max_{j\in[e]} \{ v_j + \epsilon \delta_{ij} \} 
        & = v_i + \epsilon = 
        M +\epsilon, 
        \label{eq:l1_1_a} \\ 
        \min_{j\in[e]} \{ v_j + \epsilon \delta_{ij} \} 
        & = m,
        \label{eq:l1_1_b}
    \end{align}
    and so $d_{\rm tr}(x+\epsilon E_i,y) = 
    d_{\rm tr}(x,y) + \epsilon$. 
    Note that equations \eqref{eq:l1_1_a} and \eqref{eq:l1_1_b}
    hold for all $\epsilon > 0$. 

    \item 
    If 
    $v_i = m$ {\bf and} 
    $v_i < v_k$ for all $k \neq i$, 
    i.e. if $v_i$ is the {\bf unique} minimum component of vector $v$, 
    then 
    \begin{align}
        \max_{j\in[e]} \{ v_j + \epsilon \delta_{ij} \}  = 
        M, &
        \text{ for all } 
        \epsilon \leq M-m
        \label{eq:l1_2_a}
        \\ 
        \min_{j\in[e]} \{ v_j + \epsilon \delta_{ij} \}  =
        v_i + \epsilon = 
        m + \epsilon, &
        \text{ for all } 
        \epsilon \leq m' - m,
        \label{eq:l1_2_b}
    \end{align}
    where $m' \coloneqq \min_{j: v_j > m} \{ v_j \} > m$
    is well-defined unless 
    $v_j = m$ for all $j \in [e]$ 
    i.e. for $v = m \cdot (1,\dots,1) = {\bf 0}$,
    which falls under the first case. 
    Clearly, $M \geq m'$, so for all 
    $\epsilon \in [0, m'-m]$  
    equations \eqref{eq:l1_2_a} and \eqref{eq:l1_2_b} 
    are satisfied and thence 
    $d_{\rm tr}(x+\epsilon E_i,y)=d_{\rm tr}(x,y) - \epsilon$.
    \item 
    Otherwise, if none of the first two cases hold then 
    $\exists k \neq i$ such that 
    $m = v_k \leq v_i < M$
    and so 
    \begin{align}
        \min_{j \in [e]} \{ v_j + \epsilon \delta_{ij} \} = 
        v_k = 
        m
        &, \text{ for all }
        \epsilon > 0 
        \label{eq:l1_3_a}
        \\ 
        \max_{j \in [e]} \{ v_j + \epsilon \delta_{ij} \} = 
        M
        &, \text{ if }
        \epsilon \leq M - v_i  
        \label{eq:l1_3_b}
    \end{align}
    Define $M' \coloneqq \max_{j: v_j < M} \{ v_j \}<M$ 
    which is well-defined 
    for all $v \neq {\bf 0}$ (first case). 
    Since $v_i < M$, it follows by definition that
    $v_i \leq M'$ and so
    $M - v_i \geq M - M' > 0 $. 
    As a result, for all $\epsilon \in [0, M - M']$, 
    equations \eqref{eq:l1_3_a} and \eqref{eq:l1_3_b}
    are satisfied and thence
    $d_{\rm tr}(x+\epsilon E_i,y) = d_{\rm tr}(x,y)$.
\end{enumerate}
If $v={\bf 0}$, set $\eta = + \infty$. 
Otherwise, for $v \neq {\bf 0}$, with $m',M'$ being well-defined, 
set 
\[
    \eta = \min( m'- m , M - M' ) > 0. 
\]
In all three cases and for all $\epsilon \in [0,\eta]$ the desired 
result is satisfied.
\end{proof}

\begin{lemma} 
\label{lemma:uniqueness_of_optimum}
Consider the function $q: \mathbb{R}^e/\mathbb{R}{\bf 1} \rightarrow \mathbb{R}$, 
\begin{align*}
    q(x) &= \lambda_\alpha d_{\rm tr}(x,\alpha) - \lambda_\beta d_{\rm tr}(x,\beta) 
      -\lambda_\gamma d_{\rm tr}(x,\gamma) + \lambda_\delta d_{\rm tr}(x,\delta) \\&+ 
      \log\left( 
        \frac{\lambda_\beta}{\lambda_\alpha}
      \right)
      - \log\left(  
        \frac{\lambda_\delta}{\lambda_\gamma}
       \right),
\end{align*}
where $\alpha,\beta, \gamma,\delta \in 
\mathbb{R}^e/\mathbb{R}{\bf 1}$, 
$\lambda_\alpha,\lambda_\beta, \lambda_\gamma,\lambda_\delta > 0$
and
$(\alpha,\lambda_\alpha) \neq (\beta,\lambda_\beta)$. 
A set $\mathcal{X}$ contains neighbourhoods of 
$\alpha,\beta,\gamma,\delta $. 
If $q(x)=0 ,\, \forall x\in\mathcal{X}$
then
$(\alpha,\lambda_\alpha) = (\gamma,\lambda_\gamma)$ and 
$(\beta,\lambda_\beta) = (\delta,\lambda_\delta)$.
\end{lemma}
\begin{proof}

According Lemma \ref{lemma:derivative_of_tropical_distance}, there exists 
$\eta_1 > 0$ such that for all 
$\epsilon \in [0, \eta_1]$
\begin{equation}
    d_{\rm tr}(x+\epsilon E_i,y) = d_{\rm tr}(x,y)+ \epsilon \phi_i(x-y).
    \label{eq:dtr_pos}
\end{equation}
Moreover, $d_{\rm tr}(x - \epsilon E_i, y)  = 
    d_{\rm tr}(y, x - \epsilon E_i) = 
    d_{\rm tr}(y+\epsilon E_i, x)$ and so using 
    Lemma \ref{lemma:derivative_of_tropical_distance}
    again (but with $x$ and $y$ swapped), 
    there exists $\eta_2>0$ such that for all $\epsilon \in [0,\eta_2]$
\begin{equation}
    d_{\rm tr}(x - \epsilon E_i, y)  = 
    d_{\rm tr}(x,y) + \epsilon \phi_i(y-x), 
    \label{eq:dtr_neg}
\end{equation}
for all $\epsilon \in [0,\epsilon_0(y-x)]$. 
For all $\epsilon \in [0,\eta]$ 
where $\eta \coloneqq \min(\eta_1,\eta_2)$, 
equations
\eqref{eq:dtr_pos}, \eqref{eq:dtr_neg} are satisfied and so 
\begin{align*}
    q&(x+\epsilon E_i) = q(x) + \\&\epsilon \left( 
        \lambda_\alpha \phi_i(x-\alpha) - 
        \lambda_\beta \phi_i(x-\beta) - 
        \lambda_\gamma \phi_i(x-\gamma) + 
        \lambda_\delta \phi_i(x-\delta)
    \right),  \\ 
    q&(x-\epsilon E_i) = q(x) + \\&\epsilon \left( 
        \lambda_\alpha \phi_i(\alpha-x) - 
        \lambda_\beta \phi_i(\beta-x) - 
        \lambda_\gamma \phi_i(\gamma-x) + 
        \lambda_\delta \phi_i(\delta-x) 
    \right). 
\end{align*}
Consequently, for all $\epsilon \in [0,\eta]$, 
\begin{align}
    \label{eq:h_second_diff}
    q&(x+\epsilon E_i) + q(x-\epsilon E_i) 
    - q(x)
    = 0 \\&= 
        \epsilon \left(
        \lambda_\alpha s_i(x-\alpha) - 
        \lambda_\beta s_i(x-\beta) - 
        \lambda_\gamma s_i(x-\gamma) + 
        \lambda_\delta s_i(x-\delta)
    \right), \nonumber
\end{align}
where
\begin{align}
    \label{eq:s_i}
    s_i(v) &\coloneqq \phi_i(v) + \phi_i(-v) = \\&\begin{cases}
        2, & \text{ if } v= {\bf 0} \\
        1, & \text{ if } v \neq {\bf 0} \text{ and } 
        v_i \text{ is the non-unique maximizer
        or minimizer of } v \\
        0, & \text{ otherwise }
    \end{cases}
    \nonumber
\end{align}
By summing equation \eqref{eq:h_second_diff} over $i \in [e]$
and defining $s(v) = \sum_{i=1}^e s_i(v)$,
\begin{equation}
    \label{eq:s}
     \lambda_\alpha s(x-\alpha) - 
    \lambda_\beta s(x-\beta) - 
    \lambda_\gamma s(x-\gamma) + 
    \lambda_\delta s(x-\delta) = 0, 
\end{equation}
$\forall x \in \mathcal{X}$.
\par 
Here we try to prove by contradiction that 
$\mathcal{S} := \{\alpha,\delta\}\cap \{\gamma,\beta\}$ is not empty.
Suppose that $\mathcal{S} := \{\alpha,\delta\}\cap \{\gamma,\beta\} = \emptyset$. 
Then, setting $x= \alpha$ in equation \eqref{eq:s} and noting that 
$s(0) = 2e$ and $0 \leq s(v) \leq e$ for $v \neq 0$, we get 
$2e \lambda_\alpha \leq e \lambda_\beta + e \lambda_\gamma$, since
$\beta, \gamma \neq \alpha$.
Applying the same argument to $x=\beta,\gamma,\delta$, the following system of inequalities holds
\begin{align*}
    2\lambda_\alpha &\leq \lambda_\beta + \lambda_\gamma \\
    2\lambda_\beta &\leq \lambda_\alpha + \lambda_\delta \\
    2\lambda_\gamma & \leq \lambda_\alpha + \lambda_\delta \\
    2\lambda_\delta & \leq \lambda_\beta + \lambda_\gamma.
\end{align*}
It follows that $\lambda_\alpha = \lambda_\beta = \lambda_\gamma =\lambda_\delta $.
Then, rewrite equation \eqref{eq:s} as 
\begin{equation}
    \label{eq:s2}
     s(x-\alpha) - 
     s(x-\beta) - 
     s(x-\gamma) + 
     s(x-\delta) = 0, 
\end{equation}
Note now equation \eqref{eq:s2} can only hold at $x=\alpha$ iff 
$s(\alpha - \gamma) = s(\alpha-\beta) = e$ and $s(\alpha-\delta)=0$. 
But $s(v)=e$ if and only if all the components of $v$ are non-unique minimizers and maximizers
or $\{v_i:i\in[e]\}=\{\zeta,\kappa\}$, where $\zeta < \kappa$ and 
$|\{i:v_i = \zeta \}|=n_\zeta , |\{i: v_i = \kappa \}| = n_\kappa$, 
such that $n_\zeta + n_\kappa = e$ and $n_\zeta, n_\kappa \geq 2$.
\par 
Consider $z =v+\epsilon E_i$, where $v_i = \zeta$ and $0<\epsilon < \kappa-\zeta$.
The minimum and maximum components of $z$ are 
$\zeta$ and $\kappa$, and 
$\{z_i:i\in[e]\}=\{\zeta,\zeta + \epsilon, \kappa\}$ with
$|\{i:z_i = \zeta \}|=n_\zeta - 1,|\{i: z_i = \kappa \}| = n_\kappa $. 
It follows that, 
$$s(z) = |\{i:z_i = \zeta \}| +|\{i: z_i = \kappa \}| = e-1.$$
\par 
Now consider $z = v + \epsilon E_i$ where $v_i = \kappa$. 
The maximum is no longer unique, but the $n_{\zeta}$ minima are still unique. 
Therefore, $s(z) = n_{\zeta} \geq 2$. 
Combining the two cases, it is concluded that $s(v + \epsilon E_i) \geq 2$ for all $i \in [e]$.
\par 
Set $x = \alpha + \epsilon E_i$, where $\alpha_i - \beta_i = \min_k \{ \alpha_k - \beta_k \}$. 
Then, 
\begin{equation}
    \label{eq:alpha}
    s(x - \alpha) = s(\epsilon E_i)= e-1, 
\end{equation}
since there is a unique maximizer, but all the other $e-1$ components are $0$, which is the minimum.
Furthermore, 
\begin{equation}
    \label{eq:beta}
    s(x - \beta) = s(\alpha - \beta + \epsilon E_i) = e-1, 
\end{equation}
since 
for $v = \alpha - \beta$ with $s(v) = e$, it corresponds to the first case examined.
It is assumed that $\epsilon < \kappa - \zeta = d_{\rm tr}(\alpha - \beta)$.
Moreover, 
\begin{equation}
    \label{eq:gamma}
    s(x - \gamma) = s(\alpha-\gamma + \epsilon E_i) \geq 2,
\end{equation}
for $v = \alpha - \gamma$ with $s(v) = e$. 
Finally, since $s(\alpha - \delta) = 0$ and so the components of $\alpha - \delta$ 
have a unique minimum and a unique maximum, there exists a neighborhood around 
$x = \alpha$ such that $x-\alpha$ still has that property, i.e. 
\begin{equation}
    \label{eq:delta}
    s(x - \delta) = s(\alpha - \delta + \epsilon E_i) = 0
\end{equation}
for all $\epsilon < \eta$ for some $\eta>0$.
\par 
From equations \eqref{eq:alpha} -- \eqref{eq:delta}, it is concluded that 
\begin{equation}
    \label{eq:s_contradiction}
    s(x-\alpha) - 
    s(x-\beta) - 
    s(x-\gamma) + 
    s(x-\delta) \leq -2 , 
\end{equation}
which contradicts equation \eqref{eq:s2}.
Therefore $\mathcal{S} = \{\alpha,\delta\}\cap \{\gamma,\beta\} \neq \emptyset$.
\par 
Define another set $\mathcal{T} = \{\alpha,\beta,\gamma,\delta\}$. 
Since $\mathcal{S} \neq \emptyset$, $|\mathcal{T}| \leq 3$.
Suppose that $|\mathcal{T}| = 3$ with $\mathcal{T} = \{\tau, \upsilon,\phi\}$.
Then, without loss of generality equation \eqref{eq:s} becomes 
\begin{equation}
    \label{eq:s3}
    \lambda_\tau s(x-\tau) + \lambda_\upsilon s(x- \upsilon) - \lambda_\phi s(x-\phi) = 0
\end{equation}
Similarly to before, setting $x=\tau,\upsilon,\phi$ yields,
\begin{align*}
    2\lambda_\tau &\leq \lambda_\phi \\
    2\lambda_\upsilon &\leq \lambda_\phi \\
    2\lambda_\phi & \leq \lambda_\tau + \lambda_\upsilon, 
\end{align*}
which is contradictory since $\lambda_\tau + \lambda_\upsilon>0$.
Therefore, $|\mathcal{T}|\leq 2$.
There are $4$ cases to consider
\begin{enumerate}[i.]
    \item $\alpha = \delta \neq \beta = \gamma$, but then $\mathcal{S} = \emptyset$,
    \item $\alpha = \beta \neq \gamma = \delta$, but then equation \eqref{eq:s} can only be satisfied 
    $x=\alpha,\gamma$
    if $\lambda_\alpha = \lambda_\beta$ 
    and $\lambda_\gamma  = \lambda_\delta$
    which violates the statement that 
    $(\alpha,\lambda_\alpha) \neq (\beta,\lambda_\beta)$,
    \item $\alpha = \gamma \neq \beta = \delta$ and from equation \eqref{eq:s} at $x=\alpha,\gamma$
    it follows that $\lambda_\alpha = \lambda_\gamma, \lambda_\beta = \lambda_\delta$ 
    and hence the desired result,
    \item $\alpha = \beta=\gamma=\delta$, in which case 
    \[
        q(x) = (\lambda_\alpha - \lambda_\beta - \lambda_\gamma + \lambda_\delta) d_{\rm tr}(x,\alpha)
        +\log\left( 
        \frac{\lambda_\beta}{\lambda_\alpha}
      \right)
      - \log\left(  
        \frac{\lambda_\delta}{\lambda_\gamma}
       \right),
    \]
    which can only be uniformly $0$ at $\mathcal{X}$ if and only if 
    $\lambda_\alpha + \lambda_\delta = \lambda_\beta + \lambda_\gamma$. 
    Observe that $(\lambda_\alpha,\lambda_\delta)$ and $(\lambda_\beta, \lambda_\gamma)$
    are the two roots of the same quadratic
    $z^2 - (\lambda_\alpha + \lambda_\delta) z + \lambda_\alpha \lambda_\delta$ and noting that 
    in this case
    $\lambda_\alpha \neq \lambda_\beta$, it 
    follows that
    $\lambda_\alpha = \lambda_\gamma$ and $\lambda_\beta = \lambda_\delta$. 
\end{enumerate}

\end{proof} 

\begin{lemma}
    \label{lemma:compact_set}
    Consider a compact set $\Sigma \subseteq \mathbb{R}_+ =  (0,\infty)$.
    Then the set $\Lambda = \{ \sigma^{-1}: \sigma \in \Sigma \} 
    \in \mathbb{R}_+$ is also compact.  
\end{lemma}
\begin{proof}
    In metric spaces, a set is compact iff it is sequentially compact.
    Therefore, for every sequence $\sigma_n \in \Sigma$, $\sigma_n \to \sigma \in \Sigma$.
    Every sequence in $\Lambda$ can be expressed as $1/\sigma_n$, which tends to 
    $1/\sigma \in \Lambda$. 
    Therefore, $\Lambda$ is sequentially compact and hence compact.
\end{proof}

\begin{proof}[{\bf Proof of Theorem \ref{prop:consistency_of_estimators}}]
This proof has been written for precision estimators $\lambda=1/\sigma$ instead of 
deviation estimators. 
For the rest of the proof consider $\lambda_y = \sigma_y^{-1}$ for $y=0,1$ 
and define the set 
\[\Lambda = \{ \sigma^{-1}: \sigma \in \Sigma \} 
    \in \mathbb{R}_+.\]
According to Lemma \ref{lemma:compact_set}, $\Lambda$ is also compact.\\ 
Define the function $f$ and $h$ as
\begin{align*}
    f: \trtor &\times \{0,1\} \times  \Omega^2 \times \Lambda^2 \to \R, \\
    f((x,y),(\omega,\lambda) ) & = 
    y \log S(h(x,(\omega,\lambda))) 
    + 
    (1-y) \log S(-h(x,(\omega,\lambda))), \\ 
    h: \trtor &\times \Omega^2 \times \Lambda^2  \to \R,\\
    h(x,(\omega,\lambda )) & = 
    \lambda_0 d_{\rm tr}(x,\omega_0) - \lambda_1 d_{\rm tr}(x,\omega_1)
    +(e-1) \log{\frac{\lambda_{1} }{\lambda_0}},
\end{align*} 
where $S$ is the logistic function. 
Also denote the empirical ($Q_n$) and expected ($Q$) log-likelihood functions
as 
\begin{align*}
    Q_n(\omega,\lambda) &= \frac{1}{n} \sum_{i=1}^{n} f((X_i,Y_i),(\omega,\lambda)) ~\text{ with }\\
    Q_n(\Hat{\omega}_n,\Hat{\lambda}_n) &= 
    \sup_{\omega \in \Omega^2,
    \lambda \in \Lambda^2  } Q_n(\omega),~\text{ and }\\
    Q(\omega,\lambda) &= \mathbb{E}_{(X,Y)} \left( f((X,Y),(\omega,\lambda)) \right)  \\&=
    \mathbb{E}_X(
        S(h(X,(\omega^*,\lambda^*) )) \log( S(h(X,(\omega,\lambda) ) )  \\&~~~+
        S(-h(X,(\omega^*,\lambda^*) )) \log( S(-h(X,(\omega,\lambda) ) ) 
    ).
\end{align*}
The last equation follows from conditioning on 
\[
Y \sim \text{Bernoulli}(S(h(X,(\omega^*,\lambda^*) ))).
\]
Before we move on, we need to prove that 
$f((X,Y),(\omega,\lambda))$ is integrable so that 
$Q$ is well-defined. Without loss of generality assume 
that $\lambda_1 \geq \lambda_0$. 
It suffices to prove that $\mathbb{E}(f((X,Y),(\omega,\lambda)),Y=y)$ 
is integrable for both $y=0,1$.
Observe that 
\begin{align*}
h(X,(\omega,\lambda)) &\leq 
(\lambda_0 -\lambda_1) d_{\rm tr}(X,\omega_0) +
\lambda_1 d_{\rm tr}(\omega_0,\omega_1) + 
\text{const} \\
& \leq \lambda_1 d_{\rm tr}(\omega_0,\omega_1) + 
\text{const}.
\end{align*}
Since $h(X,(\omega,\lambda))$ is bounded above, 
$f((X,Y),(\omega,\lambda))$ is also bounded below on $Y=0$ 
and is hence integral on $Y=0$. 
Also, observe that
\begin{equation*}
    h(X,(\omega,\lambda)) \geq 
(\lambda_0 -\lambda_1) d_{\rm tr}(X,\omega_1) -
\lambda_0 d_{\rm tr}(\omega_0,\omega_1) + 
\text{const} 
\end{equation*}
and noting that 
$\log(S(x)) > x-1 $ for all $x<0$
\[
    \log(S(h(X,(\omega,\lambda)))) \geq 
    h(X,(\omega,\lambda)) - 1 \geq 
    (\lambda_0 -\lambda_1) d_{\rm tr}(X,\omega_1) 
    +\text{const}.
\]
Since $d_{\rm tr}(X,\omega_1)$ is integrable on $Y=1$, 
the LHS is integrable on $Y=1$ too. 
It follows that $f(X,(\omega,\lambda))$ is 
integrable and hence $Q$ is well-defined.
\par
First, we prove that 
$Q$ is maximised at $(\omega,\lambda) = (\omega^*,\lambda^*)$ and that this 
maximizer is unique. Consider the function 
\begin{equation*}
    g: \mathbb{R} \to \mathbb{R}, \,
    g(t) = S(\alpha) \log S(t) + S(-\alpha) \log S(-t),
\end{equation*}
where $\alpha \in \mathbb{R}$ is some constant.
The function $g$
is maximised at $t=\alpha$ and applying Taylor's theorem yields 
\begin{equation*}
    g(x) = g(\alpha) - \frac{1}{2} S(\xi) S(-\xi) (x-\alpha)^2, \text{ for some }
    \xi \in (\alpha,x). 
\end{equation*}
Setting $\alpha = h(X,(\omega^*,\lambda^*) )$ and denoting $\xi$ as a random variable
$$\xi(X) \in (h(X,(\omega^*,\lambda^*)),h(X,(\omega,\lambda)))$$
observe that 
\begin{align}
    Q(\omega,\lambda) &= \mathbb{E}_X( g(h(X,(\omega,\lambda) ) ) ) \nonumber \\& 
    = \mathbb{E}_X ( g(h(X,(\omega^*,\lambda^*) ) ) - \label{eq:taylor_approx} \\&
    \frac{1}{2} 
    \mathbb{E}_X  \left(S(\xi(X)) S(-\xi(X)) [h(X,(\omega,\lambda) )-
    h(X,(\omega^*,\lambda^*) )]^2 \right)  \nonumber
    \\& \leq  Q(\omega^*,\lambda^*),   \nonumber
\end{align}
Hence, from the expression above it is deduced that $(\omega^*,\lambda^*)$
is a maximizer. 
Now, consider the function $q: \mathcal{X} \to \R $
\begin{equation*}
    q(x) = h(x,(\omega^*,\lambda^*) ) - h(x,(\omega,\lambda)),
\end{equation*}
where $\Omega \subset \mathcal{X} \subset \trtor$ such that 
for some $\zeta >0$
$$\mathcal{X} = \{x \in \trtor: \inf_{\omega \in \Omega} d_{\rm tr}(x,\omega)<\zeta \}, $$
so that for any $\omega \in \Omega$ there is a neighborhood of $\omega$ within $\mathcal{X}$.
Note that $\mathcal{X}$ is a bounded set since $\Omega$ is bounded too.
\par 
We will prove by contradiction that 
$q(x) = 0, \forall x \in \mathcal{X}$.
Suppose there exists $x_0 \in \mathcal{X}$ such that $q(x_0) > 0$, then since $q$ is continuous there exists
a neighborhood $U$ with $x_0 \in U$ such that $q(x) > 0$ for all $x \in U$ and so 
\begin{equation*}
    \mathbb{E}(q^2(X) \mathbb{I}(X \in U) ) >0, 
\end{equation*}
where $\mathbb{I}$ is the indicator function.
Since $h(x,(\omega,\lambda))$ is continuous with respect to $x$ and $\mathcal{X}$ is bounded, 
the function takes values on a bounded interval and hence $\xi(x)$ is bounded in $\mathcal{X}$
i.e. there exists $\epsilon>0$ such that
$\mathbb{P}(S(\xi(X) S(-\xi(X)) > \epsilon| X \in U) = 1$ and so equation \eqref{eq:taylor_approx} becomes
\begin{equation*}
    Q(\omega,\lambda) \leq Q(\omega^*,\lambda^*) - 
    \frac{\epsilon}{2} \mathbb{E}(q^2(X) \mathbb{I}(X \in U) )
    < Q(\omega^*,\lambda^*),
\end{equation*}
since $\mathbb{P}(X \in U) >0$ ($X$ has positive density everywhere). 
Therefore, for $(\omega,\lambda)$ to be a maximizer, 
$q(x) = 0$ for all $x \in \mathcal{X}$.
Apply Lemma \ref{lemma:uniqueness_of_optimum} 
with $\omega^* = (\alpha,\beta)$, $\omega = (\gamma,\delta)$, 
$\lambda^* = (\lambda_\alpha,\lambda_\beta)$ 
and 
$\lambda = (\lambda_\gamma,\lambda_\delta)$
with the set $\mathcal{X}$ containing neighbourhoods of $\alpha,\beta,\gamma,\delta$
and $q(x) = 0$ for all $x$ in those neighbourhoods.
It is concluded that
$\omega = \omega^*$ and $\lambda = \lambda^*$,
thus proving the uniqueness of the maximizer.
\par
Theorem \ref{thm:uniform_law_of_large_numbers}
provides the uniform law of large numbers. 
The parameter space $\Omega^2 \times \Lambda^2$ is compact since $\Omega$ 
and $\Lambda$ are compact.
Moreover, 
$f((x,y),(\omega,\lambda) )$ is clearly continuous 
at each $(\omega,\lambda) \in \Omega^2 \in \Lambda^2$. 
Finally, consider the function 
\[
    r(z) =  \sup_{\omega \in \Omega^2,\lambda \in \Lambda^2} \{ |f(z,(\omega,\lambda) )| \} = 
    - f(z, \omega(z), \lambda(z)), 
\]
since $f$ is non-positive. 
The function $\omega(z), \lambda(z)$ are chosen to
minimize $f$. Using equation \eqref{eq:taylor_approx}, 
\begin{align*}
    \mathbb{E}(r(X)) &\leq  
    - Q(\omega^*,\lambda^*) + \\&
    \frac{1}{2} 
    \mathbb{E}( \left[
        h(X,(\omega(X),\lambda(X)))- 
        h(X,(\omega^*,\lambda^*))
    \right]^2) ,
\end{align*}
since the sigmoid function is bounded by $1$. 
Note that 
\begin{equation*}
    \mathbb{E}((Z+W)^2) \leq 
     2(\mathbb{E}(Z^2) +\mathbb{E}(W^2)),
\end{equation*}
and set $W = \log(\lambda_1(X)/\lambda_0(X)) - 
\log(\lambda_1^*/\lambda_0^*)$. 
Since $\lambda_y(X) \in \Lambda \subseteq [a,b]$ for some 
$b\geq a>0$, it follows that $W^2$ is integrable and so now we just have to prove that 
$Z$ is integrable, where $Z=Z_1+Z_2+Z_3+Z_4$ with the four terms corresponding to tropical
distance function $\lambda d_{\rm tr}(X,\omega)$.
It also holds 
\begin{equation*}
    \mathbb{E}((Z_1+Z_2+Z_3+Z_4)^2) \leq 
     2(\mathbb{E}(Z_1^2) +\mathbb{E}(Z_2^2) + \mathbb{E}(Z_3^2) +\mathbb{E}(Z_4^2))
\end{equation*}
and so $\mathbb{E}(Z^2)$ is bounded above by
\begin{align*}
    \mathbb{E}\left(
    \sum_{i=0}^1 \lambda_i^2 
    d_{\rm tr}^2(X,\omega_i(X))) +
    (\lambda_i^*)^2 
    d_{\rm tr}^2(X,\omega^*_i(X))
    \right) \leq \\ 
    \mathbb{E}_Y \left[
    2 \left(
        \sum_{i=0}^1 \lambda_i^2 + (\lambda_i^*)^2
    \right)
    \mathbb{E}\left(d_{\rm tr}^2(X,\omega^*_Y)|Y \right) 
    + \right. \\ \left. 
    2 \left(
        \sum_{i=0}^1 \lambda_i^2 
        d_{\rm tr}^2( \omega_i(X),\omega_Y^*)
        + (\lambda_i^*)^2 
        d_{\rm tr}^2(\omega^*_i,\omega_Y^*)
    \right)
    \right], 
\end{align*} 
where the second inequality came from applying the 
triangular inequality four times in the form 
$d_{\rm tr}(X,\tau) \leq d_{\rm tr}(X,\omega^*_Y) +
d_{\rm tr}(\omega^*_Y,\tau)$. 
The final expression is finite because $\Omega$ is compact 
and hence $d_{\rm tr}( \omega_i(X),\omega_Y^*)$ is finite, 
$d_{\rm tr}(X,\omega_Y)|Y$ is square-integrable.
Therefore, $\mathbb{E}(r(X))$ is finite.
\par 
All conditions of the theorem are satisfied and so 
\begin{align*}
    \sup_{\omega \in \Omega^2} \left|
        \frac{1}{n} 
        \sum_{i=1}^n f( (X_i,Y_i), \omega) - 
        \mathbb{E}(f( (X,Y), \omega))
    \right| = \\ 
    \sup_{\omega \in \Omega^2} \left|
        Q_n(\omega) - Q(\omega)
    \right| \overset{p}{\to} 0.
\end{align*}
Finally, using Theorem \ref{thm:convergence_of_theta_n}
and combining the uniqueness of the maximizer with the uniform bound result, 
it is concluded that $\Hat{\omega} \overset{p}{\to} \omega^*$.
\end{proof}

%% file: main.bbl
\begin{thebibliography}{10}

\bibitem{Akian2021}
M.~Akian, S.~Gaubert, Y.~Qi, and O.~Saadi.
\newblock Tropical linear regression and mean payoff games: or, how to measure
  the distance to equilibria, 2021.
\newblock \url{https://arxiv.org/abs/2106.01930}.

\bibitem{github_aliatimis}
Georgios Aliatimis.
\newblock Tropical logistic regression.
\newblock
  \url{https://github.com/GeorgiosAliatimis/tropical_logistic_regression},
  2024.

\bibitem{Ane}
C~An\'e, B~Larget, DA~Baum, SD~Smith, and A~Rokas.
\newblock Bayesian estimation of concordance among gene trees.
\newblock {\em Mol Biol Evol.}, 24(2):412--26, 2007.

\bibitem{AK}
F.~Ardila and C.~J. Klivans.
\newblock The {B}ergman complex of a matroid and phylogenetic trees. journal of
  combinatorial theory.
\newblock {\em Series B}, 96(1):38--49, 2006.

\bibitem{bierens1996topics}
Herman~J Bierens.
\newblock {\em Topics in advanced econometrics: estimation, testing, and
  specification of cross-section and time series models}.
\newblock Cambridge University Press, 1996.

\bibitem{BHV}
L.J. Billera, S.P. Holmes, and K.~Vogtmann.
\newblock Geometry of the space of phylogenetic trees.
\newblock {\em Adv Appl Math}, 27(4):733--767, 2001.

\bibitem{Buneman}
P.~Buneman.
\newblock A note on the metric properties of trees.
\newblock {\em J. Combinatorial Theory Ser. B.}, 17:48--50, 1974.

\bibitem{joswig2023}
A.~Com\v~aneci and M.~Joswig.
\newblock Tropical medians by transportation.
\newblock {\em Math. Program}, 205:813--839, 2023.

\bibitem{criado2021tropical}
Francisco Criado, Michael Joswig, and Francisco Santos.
\newblock Tropical bisectors and {V}oronoi diagrams.
\newblock {\em Foundations of Computational Mathematics}, pages 1--38, 2021.

\bibitem{Nye2021}
M.~K. Garba, T.~M.~W. Nye, J.~Lueg, and S.~F. Huckemann.
\newblock Information geometry for phylogenetic trees.
\newblock {\em J. Math. Biol.}, 81(19), 2021.

\bibitem{huelsenbeck2001bayesian}
John~P Huelsenbeck, Fredrik Ronquist, Rasmus Nielsen, and Jonathan~P Bollback.
\newblock Bayesian inference of phylogeny and its impact on evolutionary
  biology.
\newblock {\em science}, 294(5550):2310--2314, 2001.

\bibitem{10.1093/sysbio/syr021}
P.~M. Huggins, W.~Li, D.~Haws, T.~Friedrich, J.~Liu, and R.~Yoshida.
\newblock {Bayes Estimators for Phylogenetic Reconstruction}.
\newblock {\em Systematic Biology}, 60(4):528--540, 04 2011.

\bibitem{lakner2008efficiency}
Clemens Lakner, Paul Van Der~Mark, John~P Huelsenbeck, Bret Larget, and Fredrik
  Ronquist.
\newblock Efficiency of {M}arkov chain {M}onte {C}arlo tree proposals in
  bayesian phylogenetics.
\newblock {\em Systematic biology}, 57(1):86--103, 2008.

\bibitem{LSTY}
B.~Lin, B.~Sturmfels, X.~Tang, and R.~Yoshida.
\newblock Convexity in tree spaces.
\newblock {\em SIAM Discrete Math}, 3:2015--2038, 2017.

\bibitem{lin2018tropical}
Bo~Lin and Ruriko Yoshida.
\newblock Tropical {F}ermat--{W}eber points.
\newblock {\em SIAM Journal on Discrete Mathematics}, 32(2):1229--1245, 2018.

\bibitem{MS}
D.~Maclagan and B.~Sturmfels.
\newblock {\em Introduction to Tropical Geometry}, volume 161 of {\em Graduate
  Studies in Mathematics}.
\newblock Graduate Studies in Mathematics, 161, American Mathematical Society,
  Providence, RI, 2015.

\bibitem{mesquite}
W.~P. Maddison and D.R. Maddison.
\newblock Mesquite: a modular system for evolutionary analysis. version 2.72,
  2009.
\newblock Available at \url{http://mesquiteproject.org}.

\bibitem{maddison2008mesquite}
Wayne~P Maddison.
\newblock Mesquite: a modular system for evolutionary analysis.
\newblock {\em Evolution}, 62:1103--1118, 2008.

\bibitem{newey1994large}
Whitney~K Newey and Daniel McFadden.
\newblock Large sample estimation and hypothesis testing.
\newblock {\em Handbook of econometrics}, 4:2111--2245, 1994.

\bibitem{10.1093/bioinformatics/btaa564}
Robert Page, Ruriko Yoshida, and Leon Zhang.
\newblock {Tropical principal component analysis on the space of phylogenetic
  trees}.
\newblock {\em Bioinformatics}, 36(17):4590--4598, 06 2020.

\bibitem{pin1998tropical}
Jean-Eric Pin.
\newblock Tropical semirings, 1998.

\bibitem{ronquist2005mrbayes}
Fredrik Ronquist, John~P Huelsenbeck, and Paul van~der Mark.
\newblock Mrbayes 3.1 manual, 2005.

\bibitem{SS}
D.~Speyer and B.~Sturmfels.
\newblock Tropical mathematics.
\newblock {\em Mathematics Magazine}, 82:163--173, 2009.

\bibitem{sukumaran2010dendropy}
Jeet Sukumaran and Mark~T Holder.
\newblock Dendropy: a python library for phylogenetic computing.
\newblock {\em Bioinformatics}, 26(12):1569--1571, 2010.

\bibitem{tran2020}
Ngoc Tran.
\newblock Tropical gaussians: a brief survey.
\newblock {\em Algebraic Statistics}, 11(2):155--168, 2020.

\bibitem{YBMH}
R.~Yoshida, K.~Miura, D.~Barnhill, and D.~Howe.
\newblock Tropical density estimation of phylogenetic trees, 2022.
\newblock \url{https://arxiv.org/abs/2206.04206}.

\bibitem{YTMM}
R.~Yoshida, M.~Takamori, H.~Matsumoto, and K.~Miura.
\newblock Tropical support vector machines: Evaluations and extension to
  function spaces, 2021.
\newblock \url{https://arxiv.org/abs/2101.11531}.

\bibitem{YZZ}
R.~Yoshida, L.~Zhang, and X.~Zhang.
\newblock Tropical principal component analysis and its application to
  phylogenetics.
\newblock {\em Bulletin of Mathematical Biology}, 81:568--597, 2019.

\bibitem{yoshida2023tropical}
Ruriko Yoshida, Georgios Aliatimis, and Keiji Miura.
\newblock Tropical neural networks and its applications to classifying
  phylogenetic trees.
\newblock {\em arXiv preprint arXiv:2309.13410}, 2023.

\bibitem{yoshida2022hit}
Ruriko Yoshida, Keiji Miura, and David Barnhill.
\newblock Hit and run sampling from tropically convex sets.
\newblock {\em arXiv preprint arXiv:2209.15045}, 2022.

\bibitem{YOSHIDA202377}
Ruriko Yoshida, Misaki Takamori, Hideyuki Matsumoto, and Keiji Miura.
\newblock Tropical support vector machines: Evaluations and extension to
  function spaces.
\newblock {\em Neural Networks}, 157:77--89, 2023.

\end{thebibliography}
